\documentclass[10pt]{article}

\usepackage[utf8]{inputenc}
\usepackage[T1]{fontenc}
\usepackage{amsmath,amsthm,amssymb,mathtools,mathrsfs}
\usepackage[numbers,sort&compress]{natbib}
\usepackage{enumitem}
\usepackage{microtype}
\usepackage[unicode,hidelinks]{hyperref}

\hypersetup{
  pdftitle={Cyclic Projective Orbits on Rational Normal Curves and MDS Codes},
  pdfauthor={Yangcheng Li and Pingzhi Yuan},
  pdfkeywords={MDS codes, rational normal curves, cyclic operators, companion matrices, generalized Reed--Solomon codes, algebraic tori, finite fields},
  bookmarksnumbered=true,
  bookmarksopen=true
}

\numberwithin{equation}{section}
\allowdisplaybreaks
\newtheorem{theorem}{Theorem}[section]
\newtheorem{lemma}[theorem]{Lemma}
\newtheorem{proposition}[theorem]{Proposition}
\newtheorem{corollary}[theorem]{Corollary}
\theoremstyle{definition}
\newtheorem{definition}[theorem]{Definition}
\newtheorem{example}[theorem]{Example}
\theoremstyle{remark}
\newtheorem{remark}[theorem]{Remark}

\newcommand{\F}{\mathbb F}
\newcommand{\A}{\mathbb A}
\newcommand{\PP}{\mathbb P}
\newcommand{\Gm}{\mathbb G_m}
\newcommand{\Kcode}{\mathcal K}
\newcommand{\MDSopen}{\mathscr M}
\newcommand{\GRSsurf}{\mathfrak G}
\newcommand{\Upower}{\mathfrak U}
\newcommand{\ord}{\operatorname{ord}}
\newcommand{\Sym}{\operatorname{Sym}}
\newcommand{\Disc}{\operatorname{Disc}}
\newcommand{\coeff}{\operatorname{coeff}}

\title{Cyclic Projective Orbits on Rational Normal Curves and MDS Codes}

\author{%
Yangcheng Li\thanks{Corresponding author. E-mail: \texttt{liyc@m.scnu.edu.cn}.}
\qquad
Pingzhi Yuan\thanks{E-mail: \texttt{yuanpz@scnu.edu.cn}.}\\
\small School of Mathematical Sciences, South China Normal University\\
\small Guangzhou 510631, Guangdong, P. R. China}

\date{}

\begin{document}
\baselineskip15pt
\maketitle

\begin{abstract}
Let \(A\) be a cyclic operator on an \(r\)-dimensional vector space over a
field \(k\), and let \(z\) be a cyclic vector. Their Krylov code has
parity-check matrix
\((z,Az,\ldots,A^{n-1}z)\).
For \(r\ge 3\) and \(n\ge r+3\), we prove that an MDS orbit segment lies on
a rational normal curve precisely when the projective pair \((A,[z])\) is
conjugate to one arising from the \((r-1)\)-st symmetric-power action of
\(\mathrm{PGL}_2\). Over finite fields, for companion operators, this gives a
complete classification of the generalized Reed--Solomon locus into split
semisimple, two nonsplit semisimple, and unipotent families.

Over an algebraically closed field \(k\), the Zariski closure
\(\GRSsurf_{r,k}\) of the semisimple GRS coefficient locus is an irreducible
rational surface, generically parameterized two-to-one by a two-dimensional
torus of geometric-progression root sets; reversal is the generic ambiguity.
The affine quotient of the parameter torus by reversal is the normalization of
\(\GRSsurf_{r,k}\cap D(a_0)\), its nonzero-constant-term open part. The
codimension in the space of monic degree-\(r\) polynomials is \(r-2\).
Frobenius descent gives an exact formula for the number of GRS polynomials
over \(\mathbb F_q\). A canonical remainder parity-check matrix defines the
MDS locus by a principal open condition.
For fixed \(r\ge3\) and \(n\ge r+3\), the proportion of all monic degree-\(r\)
polynomials over \(\mathbb F_q\) whose companion codes are
MDS and non-GRS tends to one as \(q\to\infty\) through prime powers.
\end{abstract}


\medskip

\section{Introduction and main results}
\label{sec:introduction}

A linear MDS code may be viewed as a projective arc, and a generalized
Reed--Solomon code corresponds to an arc contained in a rational normal
curve (RNC); see Proposition~\ref{prop:projective-grs-criterion}.  This
raises a natural rigidity question for orbit
configurations: when can a finite segment
\[
 [z],[Az],\ldots,[A^{n-1}z]
\]
of a cyclic projective orbit lie on an RNC?  The question is simultaneously
geometric, representation-theoretic, and arithmetic.  Geometrically,
\(r+2\) points in linearly general position determine a unique RNC over an
algebraic closure.  Dynamically, the shift of an orbit segment forces the
acting projectivity to preserve that curve.  Representation-theoretically,
its stabilizer is the \((r-1)\)-st symmetric-power image of
\(\operatorname{PGL}_2\).

We encode this orbit by the kernel of the Krylov matrix
\[
 H_n(A,z)=(z,Az,\ldots,A^{n-1}z),\qquad
 \Kcode_n(A,z):=\ker H_n(A,z).
\]
Here \(A\) acts on an \(r\)-dimensional vector space and \(z\) is cyclic,
meaning that \(z,Az,\ldots,A^{r-1}z\) is a basis.  The resulting Krylov code
is intrinsic: it is independent of the ambient basis and of the choice of
cyclic vector, and it depends only on the similarity class of \(A\)
(Theorem~\ref{thm:intrinsic-krylov}).  For the companion operator \(T_g\),
it is exactly the coefficient code formed by the degree-\(<n\) multiples
of \(g\) (Proposition~\ref{prop:companion-realization}).  Thus the
principal-ideal model is a later quotient-ring realization of an
operator-theoretic construction, rather than the source of the geometry.

\begin{theorem}[Orbit rigidity]
Let \(k\) be a field, let \(r\ge3\) and \(n\ge r+3\), and let \((A,z)\)
be a cyclic pair on an \(r\)-dimensional \(k\)-vector space.  Assume that
\(\Kcode_n(A,z)\) is MDS and put \(d=r-1\).  Then
\[
 [z],[Az],\ldots,[A^{n-1}z]
\]
 lies on a \(k\)-form of a rational normal curve if and only if there exist
\[
 S\in\operatorname{GL}_r(k),\qquad
 B\in\operatorname{GL}_2(k),\qquad c\in k^\times
\]
such that
\[
 SAS^{-1}=c\,\Sym^d(B)
 \quad\text{and}\quad
 [Sz]\in\nu_d\bigl(\PP^1(k)\bigr).
\]
When these conditions hold, the containing \(k\)-form is split, is unique as
a geometric RNC, and is preserved by the projective class of \(A\).
\end{theorem}
\noindent
The full statement and proof are given in
Theorem~\ref{thm:rnc-rigidity}.  Its mechanism is short but rigid: two
overlapping blocks of \(r+2=d+3\) orbit points determine the same RNC, so
the orbit shift stabilizes it; the stabilizer calculation then forces the
symmetric-power representation.  Over finite fields, uniqueness also gives
Frobenius descent, and hence geometric RNC containment is already equivalent
to GRS type (Corollary~\ref{cor:absolute-rnc}).

Although the orbit-rigidity step uses the classical uniqueness and
stabilizer properties of rational normal curves, its consequences for the
full companion family are new.  It produces a complete arithmetic
classification of the GRS locus, identifies its coefficient closure as a
rational surface, realizes the normalization of the
nonzero-constant-term part as a quotient of a two-dimensional torus by
reversal, and converts Frobenius action on the generic two-point fibers into
an exact finite-field counting formula.

For companion operators over \(\F_q\), this rigidity theorem converts the geometric
problem into the three conjugacy types of
\(\operatorname{PGL}_2(\F_q)\).  Split semisimple elements yield root sets
\(\{\gamma,\gamma t,\ldots,\gamma t^{r-1}\}\) with
\(\gamma,t\in\F_q^\times\) and \(\ord(t)\ge n\).  Nonsplit semisimple
elements diagonalize over \(\F_{q^2}\), not in general over \(\F_q\), and
Frobenius reverses the progression; this gives a product of irreducible
quadratics when \(r\) is even and one linear factor together with
irreducible quadratics when \(r\) is odd.  Finally, the unipotent type gives
\(g=(X-\beta)^r\), which is MDS precisely when
\(n\le\operatorname{char}\F_q\).  These are the four mutually exclusive
arithmetic families in Theorem~\ref{thm:arithmetic-classification}, and
together with the MDS test they give the complete trichotomy of
Theorem~\ref{thm:complete-trichotomy}.

The semisimple families admit a uniform coefficient-space description.  For
every field \(k\), the morphism
\[
 \Phi_{r,k}:\mathbb G_{m,k}^{2}\longrightarrow\A_k^r,\qquad
 (\gamma,t)\longmapsto
 \coeff\!\left(\prod_{i=0}^{r-1}(X-\gamma t^i)\right)
\]
is invariant under reversal
\(\iota(\gamma,t)=(\gamma t^{r-1},t^{-1})\).  Over an algebraically
closed field \(k\), its image closure \(\GRSsurf_{r,k}\) is an irreducible
rational surface, the generic fiber of \(\Phi_{r,k}\) is exactly an
\(\iota\)-pair, and
\[
 \dim\GRSsurf_{r,k}=2,\qquad
 \operatorname{codim}_{\A_k^r}\GRSsurf_{r,k}=r-2
\]
(Theorem~\ref{thm:coefficient-surface}).  More precisely, the affine
categorical quotient
\[
 \mathbb G_{m,k}^{2}/\!/\langle\iota\rangle
\]
is the normalization of
\(\GRSsurf_{r,k}\cap D(a_0)\), in every characteristic
(Theorem~\ref{thm:normalization-open}).  The collision boundary cannot be
silently included: for every algebraically closed field \(k\), the exact
formula is
\[
 \operatorname{GRS}_{n,r}(k)
 =
 \bigl(\MDSopen_{n,r}\cap\GRSsurf_{r,k}\cap D(\Disc)\bigr)(k)
 \;\dot\cup\;
 \bigl(\MDSopen_{n,r}\cap\Upower_{r,k}^{\times}\bigr)(k),
\]
where \(\Upower_{r,k}^{\times}\) is the nonzero-constant-term part of the
closed pure-power curve \(\Upower_{r,k}^{\mathrm{cl}}\)
(Theorem~\ref{thm:geometric-grs-locus}).
For cubics this becomes the exact finite-field test
\(D_{n,3}(g)\ne0\) together with \(a_0a_2^3=a_1^3\)
(Corollary~\ref{cor:exact-cubic-grs}).
This two-dimensional locus is
forced by the cyclic orbit; it is not the parameter space of arbitrary
ordered point configurations on an RNC studied in
\cite{CaminataEtAl2018}, and we do not identify the two spaces.

The generic reversal ambiguity also makes the finite-field locus exactly
countable.
For a cyclic group \(C_M\), set
\[
 E_M(n)=\#\{t\in C_M:\ord(t)\ge n\}
       =\sum_{\substack{e\mid M\\e\ge n}}\varphi(e).
\]
If \(\mathcal G_{n,r}(q)\) denotes the monic degree-\(r\) polynomials whose
companion codes are MDS and of GRS type, then, for \(r\ge3\),
\(n\ge r+3\), and \(p=\operatorname{char}\F_q\),
\[
 |\mathcal G_{n,r}(q)|
 =
 (q-1)\mathbf 1_{\{n\le p\}}
 +\frac{q-1}{2}\bigl(E_{q-1}(n)+E_{q+1}(n)\bigr)
\]
(Theorem~\ref{thm:exact-grs-count}).  In contrast, for fixed \(r,n\) in
the same stable range, the number of monic \(g\in\F_q[X]\) defining MDS
non-GRS codes is
\[
 q^r+O_{n,r}(q^{r-1})
\]
as \(q\) tends to infinity through prime powers
(Theorem~\ref{thm:generic-nongrs}).  Thus MDS non-GRS members have density
one, whereas the GRS family remains two-dimensional.

The open MDS condition supporting both the count and the trichotomy is
canonical.  If \(Q_i\) is the coefficient vector of the remainder of
\(X^i\) modulo \(g\), then
\(H_g^{\mathrm{rem}}=(Q_0\ Q_1\ \cdots\ Q_{n-1})\) is a parity-check
matrix.  Translation along the companion orbit reduces all maximal minors
to those containing \(Q_0\); their product is a polynomial \(D_{n,r}\), and
the code is MDS exactly on the principal open set
\(\MDSopen_{n,r}=D(D_{n,r})\subseteq\A^r\)
(Theorem~\ref{thm:determinantal-mds}).  This gives both an exact recognition
test and the uniform finite-field estimate used in asymptotic genericity.

The work sits at the intersection of three established lines.  In finite
projective geometry, the relevant background on arcs and RNCs is surveyed in
\cite{Ball-Lavrauw,Hirschfeld-Thas}; our rigidity argument uses the
\(d+3\)-point form of Castelnuovo's lemma
\cite[Theorem~1.18]{Harris1992}, while equations for RNC point
configurations provide a broader moduli-theoretic setting
\cite{CaminataEtAl2018}.

Two meanings of ``cyclic'' should be separated.  Here a cyclic pair
\((A,z)\) means that \(z\) is a cyclic vector for the linear operator
\(A\); the columns of \(H_n(A,z)\) form only the marked finite segment
\(z,Az,\ldots,A^{n-1}z\).  We do not assume that the resulting code is
invariant under the length-\(n\) cyclic shift of its coordinates.  The
latter is the usual coding-theoretic notion of a cyclic code, studied for
MDS codes by Shokrollahi \cite{Shokrollahi2000} and in related cyclic,
constacyclic, and polycyclic settings in
\cite{Ball-Grassl-Rotteler,Liu-Liu-constacyclic,Shi-Li-Sepasdar-Sole}.

There is also a neighboring finite-geometric literature.  Maruta studies
cyclic arcs, in the sense adopted there, and their equivalent pseudo-cyclic
MDS codes \cite{Maruta1997}.  Our orbit arc need not close up to a complete
orbit of a finite cyclic group and is not assumed to be setwise invariant;
it may be an arbitrary finite initial segment of the projective orbit of
\([A]\).  More recently, Chen, Huang, and Wu classified regular cyclic
\((q+1)\)-arcs in \(\mathrm{PG}(3,2^m)\), for \(m\ge 3\), and established a
descent criterion for their diagonal models \cite{ChenHuangWu2025}.  Their
main classification therefore has vector-space dimension \(r=4\), even
field size \(q=2^m\), full length \(n=q+1\), and a regular cyclic group of
that order.  The stable regime considered here allows general \(r\) and
finite lengths \(n\ge r+3\) over arbitrary finite fields, without a
full-orbit hypothesis.  Thus the two settings have related projective-orbit
and descent questions but different global hypotheses.

In coding theory, the original Reed--Solomon construction and its standard
GRS formulation \cite{Reed-Solomon,MacWilliams-Sloane}, automorphism groups
\cite{Joyner-Ksir-Traves}, and counting
\cite{Beelen-Glynn-Hoholdt-Kaipa} motivate the coding layer of our
formulation.  Twisted codes \cite{Beelen-Puchinger-Rosenkilde} and recent
constructions or recognition methods
\cite{AbdukhalikovDingVerma2026,Jin-Ma-Xing-Zhou,LiuLiuOggier2026,
WangLiuLuo2026} address arcs outside RNCs.  Within the companion-operator
family considered here, we identify the GRS subvariety and prove that its
MDS non-GRS complement is asymptotically generic.

The principal contributions are the following.
\begin{enumerate}[leftmargin=2.4em,label=\textup{(\arabic*)}]
\item We develop an intrinsic cyclic-pair theory of Krylov codes and prove
the RNC rigidity theorem: in the stable range, RNC containment is
equivalent to membership in the symmetric-power image of
\(\operatorname{PGL}_2\).
\item We classify the finite-field companion specializations by the split
semisimple, nonsplit semisimple, and unipotent projective types, including
the parity-sensitive Frobenius reversal in the nonsplit case.
\item We construct the two-dimensional coefficient surface, identify the
normalization of its nonzero-constant-term part as the affine torus quotient
by reversal, determine the correct collision-free GRS locus, and derive
Frobenius descent together with the exact counting formula.
\item We identify the determinantal MDS open set and combine it with the
coefficient surface to obtain a computable MDS/GRS trichotomy and the
asymptotic genericity of MDS non-GRS companion codes.
\end{enumerate}

The paper is organized as follows.  Section~\ref{sec:cyclic-pairs}
develops cyclic pairs and intrinsic Krylov codes.
Section~\ref{sec:mds-locus} constructs the determinantal MDS open set and
its remainder and root-jet models.  Section~\ref{sec:rnc-rigidity} proves
the orbit-rigidity theorem, and Section~\ref{sec:finite-field-classification} gives the
finite-field arithmetic classification.  Section~\ref{sec:coefficient-surface}
studies the coefficient surface and its normalization, the collision boundary,
Frobenius descent, exact counting, and asymptotic genericity.
Section~\ref{sec:principal-ideal-application} recovers the
principal-ideal realization and presents three reproducible examples.
Section~\ref{sec:conclusion} records the conclusions and remaining open
problems.
Appendix~\ref{app:boundary-cases} treats the boundary ranges, and
Appendix~\ref{app:stabilizer} supplies the stabilizer and
symmetric-power calculations.
 \section{Cyclic pairs and Krylov codes}
\label{sec:cyclic-pairs}

Throughout, \(k\) is a field and \(V\) is an \(r\)-dimensional
\(k\)-vector space.  The coefficient weight of a polynomial is the number
of its nonzero coefficients.  We use column-vector conventions; changing
those conventions only transposes the companion matrices below.

\subsection{Cyclic operators and independence of the cyclic vector}

\begin{definition}[Cyclic pair and Krylov code]
\label{def:cyclic-pair}
Let \(A\in\operatorname{End}_k(V)\).  A vector \(z\in V\) is
\emph{cyclic} for \(A\) if
\[
 z,Az,\ldots,A^{r-1}z
\]
is a basis of \(V\).  In this case we call \((A,z)\) a cyclic pair.  For
\(n>r\), after choosing any basis of \(V\), set
\[
 H_n(A,z):=(z,Az,\ldots,A^{n-1}z)
\]
and
\[
 \mathcal K_n(A,z):=\ker H_n(A,z)\subseteq k^n.
\]
We call \(\mathcal K_n(A,z)\) the \emph{Krylov code} of the cyclic pair.
\end{definition}

The apparent choices in this construction disappear intrinsically.

\begin{theorem}[Intrinsic nature of the Krylov code]
\label{thm:intrinsic-krylov}
Let \(A\) be cyclic on \(V\), let \(z\) be cyclic, and let \(m_A\) be the
minimal polynomial of \(A\).  Then:
\begin{enumerate}[label=\textup{(\roman*)}]
\item \(\mathcal K_n(A,z)\) is independent of the basis of \(V\);
\item if \(w\) is any other cyclic vector for \(A\), then
      \(\mathcal K_n(A,w)=\mathcal K_n(A,z)\);
\item if \(A'=SAS^{-1}\) and \(z'=Sz\) for some
      \(S\in\operatorname{GL}(V)\), then
      \(\mathcal K_n(A',z')=\mathcal K_n(A,z)\);
\item the code is determined by the similarity class of \(A\), and hence,
      for cyclic \(A\), by the monic polynomial \(m_A\).
\end{enumerate}
\end{theorem}

\begin{proof}
A change of basis in \(V\) left-multiplies \(H_n(A,z)\) by an invertible
matrix and therefore does not change its kernel.  Since \(z\) is cyclic,
the map
\[
 k[X]/(m_A)\longrightarrow V,\qquad f\longmapsto f(A)z,
\]
is an isomorphism of \(k[X]\)-modules.  Thus every \(w\in V\) has the
form \(f(A)z\).  It is cyclic precisely when the class of \(f\) is a unit
in \(k[X]/(m_A)\), or equivalently when \(\gcd(f,m_A)=1\).  In that case
\(f(A)\) is invertible and commutes with \(A\), and hence
\[
 H_n(A,w)
 =\bigl(f(A)z,Af(A)z,\ldots,A^{n-1}f(A)z\bigr)
 =f(A)H_n(A,z).
\]
This proves the second assertion.  Similarly,
\[
 H_n(SAS^{-1},Sz)=S H_n(A,z),
\]
which proves simultaneous-similarity invariance.  Finally, a cyclic
operator has one invariant factor, namely \(m_A\), so cyclic operators
with the same monic minimal polynomial are similar.
\end{proof}

We next isolate the MDS condition before imposing companion coordinates.

\begin{theorem}[MDS criterion for a cyclic pair]
\label{thm:krylov-mds}
Let \(A\) be cyclic on \(V\), let \(z\) be cyclic, and let \(n>r\).  Then
\(\mathcal K_n(A,z)\) has dimension \(n-r\).  The following conditions
are equivalent:
\begin{enumerate}[label=\textup{(\roman*)}]
\item \(\mathcal K_n(A,z)\) is an \([n,n-r,r+1]\) MDS code;
\item every \(r\) columns of \(H_n(A,z)\) are linearly independent;
\item there is no nonzero polynomial
      \[
      f(X)=\sum_{i=0}^{n-1}c_iX^i
      \]
      having coefficient weight at most \(r\) and satisfying \(m_A\mid f\).
\end{enumerate}
If these conditions hold, then \(A\) is invertible and
\[
 [z],[Az],\ldots,[A^{n-1}z]\subseteq\PP(V)
\]
is an \(n\)-arc.
\end{theorem}

\begin{proof}
The first \(r\) columns form a basis, so \(H_n(A,z)\) has rank \(r\) and
its kernel has dimension \(n-r\).  The equivalence of the first two
conditions is the parity-check characterization of an MDS code.  A
dependence among at most \(r\) columns is the same as a nonzero polynomial
\(f\) of degree \(<n\) and coefficient weight at most \(r\) for which
\(f(A)z=0\).  The annihilator of the cyclic vector \(z\) is the ideal
\((m_A)\), so this happens precisely when \(m_A\mid f\).  If the dependent
set has fewer than \(r\) columns, then \(n>r\) allows us to adjoin distinct
columns until it has size \(r\); the enlarged set remains dependent.
Consequently, the absence of a dependence among exactly \(r\) columns is
equivalent to the absence of one among at most \(r\) columns.

If the code is MDS, then \(Az,A^2z,\ldots,A^rz\) are independent.  They
are the images under \(A\) of the basis \(z,Az,\ldots,A^{r-1}z\), so
\(A\) is invertible.  The assertion about the projective orbit is exactly
the nonvanishing of all \(r\)-column minors.
\end{proof}

\subsection{Companion realization}

Let
\[
 g(X)=X^r+a_{r-1}X^{r-1}+\cdots+a_0\in k[X]
\]
be monic.  Write \(T_g\) for multiplication by \(X\) on \(k[X]/(g)\),
and let \(\bar 1\) be the residue class of \(1\).

\begin{proposition}[Companion realization]
\label{prop:companion-realization}
The pair \((T_g,\bar 1)\) is cyclic, and both the minimal and
characteristic polynomials of \(T_g\) are \(g\).  Moreover,
\[
 \mathcal K_n(T_g,\bar 1)
 =\left\{(c_0,\ldots,c_{n-1})\in k^n:
     g(X)\mid\sum_{i=0}^{n-1}c_iX^i\right\}.
\]
Under coefficient-vector identification this is
\[
 \mathcal K_n(T_g,\bar 1)
 =\{u(X)g(X):\deg u<n-r\}.
\]
We henceforth abbreviate this companion Krylov code by
\(\mathcal C_g\).
\end{proposition}

\begin{proof}
The classes \(\bar1,\bar X,\ldots,\overline{X^{r-1}}\) form a basis,
so \(\bar1\) is cyclic.  Its annihilator is exactly \((g)\); therefore
the minimal polynomial is \(g\), and equality of degrees gives the same
characteristic polynomial.  For \(c=(c_0,\ldots,c_{n-1})\),
\[
 H_n(T_g,\bar1)c^{\mathsf T}
 =\sum_{i=0}^{n-1}c_i\overline{X^i}.
\]
This vanishes exactly when \(g\) divides the displayed polynomial.  Since
that polynomial has degree \(<n\), every such multiple is uniquely
\(u(X)g(X)\) with \(\deg u<n-r\).
\end{proof}

The quotient-ring principal-ideal model will be recovered only in
Section~\ref{sec:principal-ideal-application}.  Proposition
\ref{prop:companion-realization} shows that the code itself depends on the
companion operator, not on an auxiliary ambient modulus.

\subsection{The projective orbit interpretation}

An \(n\)-arc in \(\PP^{r-1}(k)\) is an ordered set of \(n\) points such
that every \(r\) of them span the ambient projective space.  Thus
Theorem~\ref{thm:krylov-mds} turns MDS recognition into a question about
the finite orbit segment
\[
 [z],[Az],\ldots,[A^{n-1}z].
\]

For \(1\le m<n\), let
\[
 P_i=[s_i:t_i]\in\PP^1(\F_q)
 \qquad(1\le i\le n)
\]
be pairwise distinct, and let \(v_i\in\F_q^\times\).  The homogeneous
extended generalized Reed--Solomon code
\(\operatorname{EGRS}_m(P,\boldsymbol v)\) is the row space of
\[
 \left(
 v_i
 \begin{pmatrix}
 s_i^{m-1}\\ s_i^{m-2}t_i\\ \vdots\\ t_i^{m-1}
 \end{pmatrix}
 \right)_{1\le i\le n}.
\]
Changing the representative \((s_i,t_i)\) only changes \(v_i\).
The point \([0:1]\) gives the usual extended coordinate at infinity.
We say that an \([n,m]_q\) code is \emph{of GRS type} if it is monomially
equivalent to such a code.

\begin{proposition}[Projective characterization of GRS type]
\label{prop:projective-grs-criterion}
Let \(\mathcal C\) be an \([n,n-r]_q\) code with \(2\le r<n\), and let
\[
 H=(h_1\ \cdots\ h_n)\in\operatorname{Mat}_{r\times n}(\F_q)
\]
be a full-rank parity-check matrix.  Assume that the columns \(h_i\) are
nonzero and pairwise nonproportional; this is automatic when
\(\mathcal C\) is MDS.  The following conditions are equivalent.
\begin{enumerate}[label=\textup{(\roman*)}]
\item \(\mathcal C\) is of GRS type.
\item The points \([h_1],\ldots,[h_n]\) lie on a split rational normal
      curve (RNC) in \(\PP^{r-1}_{\F_q}\), that is, on an
      \(\F_q\)-projective image of
      \(\nu_{r-1}(\PP^1_{\F_q})\).
\item The points \([h_1],\ldots,[h_n]\) lie on an \(\F_q\)-form of an
      RNC, namely a closed \(\F_q\)-subscheme whose base change to
      \(\overline{\F}_q\) is an RNC.
\end{enumerate}
\end{proposition}

\begin{proof}
If \(\mathcal C\) is MDS, every \(r\) columns of \(H\) are independent.
A zero column or a proportional pair could be enlarged to a set of \(r\)
columns because \(n>r\), and that enlarged set would remain dependent.
This proves the parenthetical assertion about the column hypotheses.

We first record extended GRS duality in the homogeneous convention above.
Let \(m+r=n\).  After ordering the points, every point except possibly the
last has the form \([1:a_i]\); the possible last point is
\(\infty=[0:1]\).  For the finite points put
\[
 L_i=\prod_{\substack{j\ne i\\P_j\ne\infty}}(a_i-a_j).
\]
Lagrange interpolation gives
\begin{equation}
 \sum_{P_i\ne\infty}\frac{a_i^d}{L_i}
 =
 \begin{cases}
 0,&0\le d\le N-2,\\
 1,&d=N-1,
 \end{cases}
 \label{eq:lagrange-leading-coefficient}
\end{equation}
where \(N\) is the number of finite points.  If no point is at infinity,
choose \(u_i=(v_iL_i)^{-1}\).  The inner product of row \(\alpha\) of the
dimension-\(m\) evaluation matrix with row \(\beta\) of the
dimension-\(r\) matrix with multipliers \(\boldsymbol u\) is the sum in
\eqref{eq:lagrange-leading-coefficient} with
\(d=\alpha+\beta\le n-2=N-2\), and hence is zero.

If \(P_n=\infty\), use the same finite multipliers and set
\(u_n=-v_n^{-1}\).  Now \(N=n-1\).  The finite contribution is zero unless
\(\alpha=m-1\) and \(\beta=r-1\), in which case it is \(1\); the two
columns at infinity contribute only in this last case and contribute
\(v_nu_n=-1\).  Thus the two row spaces are orthogonal in both cases.
Their dimensions sum to \(n\), so
\begin{equation}
 \operatorname{EGRS}_m(P,\boldsymbol v)^\perp
 =
 \operatorname{EGRS}_r(P,\boldsymbol u)
 \label{eq:extended-grs-duality}
\end{equation}
for suitable nonzero multipliers \(\boldsymbol u\).

Suppose \textup{(i)} holds.  Taking the dual of a monomial equivalence
again gives a monomial equivalence, with the inverse-transpose monomial
matrix.  By \eqref{eq:extended-grs-duality},
\(\mathcal C^\perp\) therefore has a generator matrix obtained from an
extended GRS generator matrix of dimension \(r\) by column permutation,
nonzero column scaling, and invertible row operations.  Permutations and
column scalings do not change the underlying projective point set, while
row operations apply an \(\F_q\)-projectivity.  Hence the projective
columns of \(H\) lie on a split RNC, proving \textup{(ii)}.

Conversely, assume \textup{(ii)}.  Choose an \(\F_q\)-projectivity carrying
the containing curve to the standard RNC and lift it to
\(S\in\operatorname{GL}_r(\F_q)\).  For each \(i\), there are
\(P_i=[s_i:t_i]\in\PP^1(\F_q)\) and \(u_i\in\F_q^\times\) such that
\[
 Sh_i
 =
 u_i
 \begin{pmatrix}
 s_i^{r-1}\\s_i^{r-2}t_i\\\vdots\\t_i^{r-1}
 \end{pmatrix}.
\]
The points \(P_i\) are distinct because the \(h_i\) are pairwise
nonproportional and the Veronese map is a closed immersion.  Thus \(SH\)
generates \(\operatorname{EGRS}_r(P,\boldsymbol u)=\mathcal C^\perp\).
Applying \eqref{eq:extended-grs-duality} once more shows that
\(\mathcal C\) is an extended GRS code with suitable multipliers, proving
\textup{(i)}.

The implication \textup{(ii)}\(\Rightarrow\)\textup{(iii)} is immediate.
Under \textup{(iii)}, the containing form has the rational point
\([h_1]\).  If \(r=2\), a degree-one RNC is the whole projective line;
faithfully flat descent from \(\overline{\F}_q\) therefore makes the form
equal to \(\PP^1_{\F_q}\).  If \(r\geq3\),
Lemma~\ref{lem:rnc-rational-point-split} shows that the form is split.
Thus \textup{(ii)} holds in either case, completing the proof.
\end{proof}

For an MDS Krylov code, Theorem~\ref{thm:krylov-mds} supplies exactly the
nonzero and pairwise nonproportional hypotheses in
Proposition~\ref{prop:projective-grs-criterion}.  Thus the code is of GRS
type precisely when its projective orbit arc lies on an \(\F_q\)-form of
an RNC, which is then necessarily split.  The central geometric problem is
therefore to decide when an MDS segment of a cyclic projective orbit lies on
an RNC.  Section~\ref{sec:rnc-rigidity} answers this intrinsically.
 \section{The MDS locus}
\label{sec:mds-locus}

The intrinsic sparse-multiple criterion of
Theorem~\ref{thm:krylov-mds} becomes a polynomial condition on companion
coefficients.  This section records the coordinate models used later for
coefficient geometry and finite-field estimates.

\subsection{Remainder-coordinate parity checks}

Let
\[
 g(X)=X^r+a_{r-1}X^{r-1}+\cdots+a_1X+a_0
\]
be the universal monic degree-\(r\) polynomial.  For \(i\ge0\), write
\[
 X^i\equiv A_{0,i}+A_{1,i}X+\cdots+A_{r-1,i}X^{r-1}\pmod g
\]
and define
\[
 Q_i=(A_{0,i},A_{1,i},\ldots,A_{r-1,i})^{\mathsf T}.
\]
Thus \(Q_0,\ldots,Q_{r-1}\) are the standard basis and
\begin{equation}
 Q_{i+r}=-a_{r-1}Q_{i+r-1}-\cdots-a_1Q_{i+1}-a_0Q_i.
 \label{eq:remainder-recurrence}
\end{equation}
Every entry of \(Q_i\) is consequently an integral polynomial in
\(a_0,\ldots,a_{r-1}\).

\begin{proposition}[Remainder-coordinate parity checks]
\label{prop:remainder-matrix}
For a specialization of \(g\) over any field \(k\), the matrix
\[
 H_g^{\mathrm{rem}}=(Q_0\ Q_1\ \cdots\ Q_{n-1})
\]
is the Krylov parity-check matrix
\(H_n(T_g,\bar1)\) in the residue basis
\(\bar1,\bar X,\ldots,\overline{X^{r-1}}\).  In particular, its kernel
is the coefficient code of the degree-\(<n\) multiples of \(g\).
\end{proposition}

\begin{proof}
The vector \(Q_i\) is the coordinate vector of
\(T_g^i\bar1=\overline{X^i}\).  The assertion therefore follows from
Proposition~\ref{prop:companion-realization}.
\end{proof}

\subsection{Sparse multiples and the determinantal open set}

For
\[
 U=\{u_1<\cdots<u_{r-1}\}\subseteq\{1,\ldots,n-1\},
\]
put
\[
 \delta_U(g)=\det(Q_0,Q_{u_1},\ldots,Q_{u_{r-1}})
\]
and define
\begin{equation}
 D_{n,r}(g)
 =a_0\prod_{\substack{U\subseteq\{1,\ldots,n-1\}\\|U|=r-1}}
 \delta_U(g)
 \in\mathbb Z[a_0,\ldots,a_{r-1}].
 \label{eq:mds-determinant}
\end{equation}
For \(r=1\) we use \(\delta_\varnothing=1\).  The factor \(a_0\)
records invertibility of the companion operator.

\begin{theorem}[The determinantal MDS open set]
\label{thm:determinantal-mds}
Let \(1\le r<n\), and specialize \(g\) over a field \(k\).  The following
are equivalent:
\begin{enumerate}[label=\textup{(\roman*)}]
\item the companion Krylov code is an \([n,n-r,r+1]\) MDS code;
\item every \(r\) columns of \(H_g^{\mathrm{rem}}\) are independent;
\item no nonzero polynomial \(f\) satisfies
      \[
      \deg f<n,\qquad g\mid f,\qquad
      \operatorname{wt}_{\mathrm{coeff}}(f)\le r;
      \]
\item \(D_{n,r}(g)\ne0\).
\end{enumerate}
Consequently the MDS coefficient locus is the principal open subset
\[
 \mathscr M_{n,r}:=D(D_{n,r})\subseteq\A^r.
\]
\end{theorem}

\begin{proof}
The first three conditions are the companion specialization of
Theorem~\ref{thm:krylov-mds}.  It remains to justify the reduced product.
If \(a_0=0\), then \(g\), viewed as a vector of length \(n\), is a codeword
of weight at most \(r\), so the code is not MDS.  Suppose \(a_0\ne0\).
Then \(T_g\) is invertible and \(Q_i=T_g^iQ_0\).  For
\(0\le i_0<i_1<\cdots<i_{r-1}<n\), applying \(T_g^{-i_0}\) to every
column gives
\begin{equation}
 \det(Q_{i_0},\ldots,Q_{i_{r-1}})
 =\det(T_g)^{i_0}
   \det(Q_0,Q_{i_1-i_0},\ldots,Q_{i_{r-1}-i_0}).
 \label{eq:minor-shift}
\end{equation}
Hence every maximal minor is nonzero exactly when all the reduced minors
appearing in \eqref{eq:mds-determinant} are nonzero.  This is equivalent
to \(D_{n,r}(g)\ne0\).
\end{proof}

\begin{remark}
The polynomial \(D_{n,r}\) is deliberately not expanded.  Its product
description is sufficient for geometry, gives an exact finite recognition
test, and avoids large coefficient recurrences that belong to the
supplementary material.
\end{remark}

\subsection{Root jets and confluent Vandermonde matrices}

For \(F(X)=\sum_i c_iX^i\), write
\[
 D^{(j)}F(Y)=\sum_{i\ge j}\binom{i}{j}c_iY^{i-j}
\]
for the \(j\)-th Hasse derivative.  Let \(L\) be a splitting field of a
specialization of \(g\), and write
\[
 g(X)=\prod_{\nu=1}^{s}(X-\lambda_\nu)^{m_\nu},
 \qquad m_1+\cdots+m_s=r,
\]
with distinct roots \(\lambda_\nu\).  Define the \(r\times n\) jet matrix
\[
 H_g^{\mathrm{jet}}
 =\left(\binom{i}{j}\lambda_\nu^{\,i-j}\right)_
 {\substack{1\le\nu\le s,\ 0\le j<m_\nu\\0\le i<n}},
\]
where the entry is \(0\) when \(i<j\).

\begin{theorem}[Root-jet form of the MDS criterion]
\label{thm:root-jet-mds}
After scalar extension to \(L\), the jet matrix is obtained from
\(H_g^{\mathrm{rem}}\) by an invertible row transformation.  Thus the
companion Krylov code is MDS if and only if every \(r\times r\) minor of
\(H_g^{\mathrm{jet}}\) is nonzero.
\end{theorem}

\begin{proof}
The identity
\[
 F(\lambda+T)=\sum_{j\ge0}D^{(j)}F(\lambda)T^j
\]
shows that \((X-\lambda)^m\mid F\) exactly when the first \(m\) Hasse
derivatives of \(F\) vanish at \(\lambda\).  Hence the Hermite--Hasse map
\[
 \mathcal J_g:L[X]_{<r}\longrightarrow L^r,\qquad
 F\longmapsto
 \bigl(D^{(j)}F(\lambda_\nu)\bigr)_{\nu,j},
\]
is injective: an element of its kernel is divisible by \(g\) and has degree
less than \(r\).  It is therefore an isomorphism.  Let \(B\) be its matrix
in the monomial basis.  The remainder of \(X^i\) modulo \(g\) has the same
prescribed jets as \(X^i\), so
\[
 H_g^{\mathrm{jet}}=B H_g^{\mathrm{rem}}.
\]
The matrix \(B\) is invertible, and scalar extension does not change
whether a determinant defined over \(k\) vanishes.  The result follows from
Theorem~\ref{thm:determinantal-mds}.
\end{proof}

The squarefree case is a generalized Vandermonde matrix; for background on
generalized and confluent variants, see \cite{Flowe-Harris,Li-Lin}.  In
particular, the geometric-progression specialization used repeatedly below
is immediate.

\begin{lemma}[Semisimple progression orbits are MDS]
\label{lem:semisimple-orbit-mds}
Suppose that over a splitting field the roots of \(g\) can be ordered as
\[
 \gamma,\gamma t,\ldots,\gamma t^{r-1},
\]
where \(\gamma,t\ne0\) and \(\ord(t)\ge n\), with infinite order allowed.
Then the companion Krylov code is MDS.
\end{lemma}

\begin{proof}
For \(0\le i_1<\cdots<i_r<n\), the corresponding root-evaluation minor is
\[
 \det\bigl((\gamma t^j)^{i_b}\bigr)_{0\le j<r,\,1\le b\le r}
 =\gamma^{i_1+\cdots+i_r}
   \prod_{1\le a<b\le r}(t^{i_b}-t^{i_a}).
\]
Every factor is nonzero because
\(0<i_b-i_a<n\le\ord(t)\).  Apply
Theorem~\ref{thm:root-jet-mds}.
\end{proof}

\begin{proposition}[Characteristic-uniform nontriviality of the MDS
polynomial]
\label{prop:mds-polynomial-nonzero}
For every \(1\le r<n\) and every prime \(p\), the reduction of
\(D_{n,r}\) modulo \(p\) is a nonzero polynomial in
\(\F_p[a_0,\ldots,a_{r-1}]\).
\end{proposition}

\begin{proof}
Choose an integer \(M\ge n\) with \(p\nmid M\).  There is an integer
\(e\ge1\) such that \(M\mid p^e-1\); choose
\(t\in\F_{p^e}^\times\) of multiplicative order \(M\).  Over
\(\F_{p^e}\), put
\[
 g(X)=\prod_{i=0}^{r-1}(X-t^i).
\]
Lemma~\ref{lem:semisimple-orbit-mds} shows that the companion Krylov code
of \(g\) is MDS.  Theorem~\ref{thm:determinantal-mds} therefore gives
\[
 D_{n,r}(g)\ne0.
\]
Thus the reduction of \(D_{n,r}\) cannot be the zero polynomial.  The
witnessing coefficient point need not be \(\F_p\)-rational: a single nonzero
evaluation over the finite extension \(\F_{p^e}\) already proves that the
reduction of \(D_{n,r}\) is not the zero polynomial over \(\F_p\).
\end{proof}

\begin{lemma}[Finite-field polynomial zero bound]
\label{lem:finite-field-zero-bound}
Let \(F\in\F_q[Y_1,\ldots,Y_s]\) be nonzero of total degree at most
\(\delta\).  Then
\[
 \#\{y\in\F_q^s:F(y)=0\}\le \delta q^{s-1}.
\]
\end{lemma}

\begin{proof}
We argue by induction on \(s\).  For \(s=1\), a nonzero polynomial of
degree at most \(\delta\) has at most \(\delta\) roots.  Suppose
\(s>1\), and write
\[
 F=\sum_{j=0}^{h}F_j(Y_1,\ldots,Y_{s-1})Y_s^j,
 \qquad F_h\ne0.
\]
Here \(h\le\delta\) and \(\deg F_h\le\delta-h\).  By induction, the set
\[
 Z_h=\{y'\in\F_q^{s-1}:F_h(y')=0\}
\]
has size at most \((\delta-h)q^{s-2}\).  For \(y'\notin Z_h\), the
specialization \(F(y',Y_s)\) has at most \(h\) roots; for
\(y'\in Z_h\), it has at most \(q\) roots.  Consequently,
\[
 \begin{aligned}
 \#V(F)(\F_q)
 &\le h\bigl(q^{s-1}-|Z_h|\bigr)+q|Z_h|\\
 &\le hq^{s-1}+(\delta-h)q^{s-1}
 =\delta q^{s-1}.
 \end{aligned}
\]
\end{proof}
 \section{Rigidity on rational normal curves}
\label{sec:rnc-rigidity}

Let \(d\ge2\).  The standard rational normal curve (RNC) in
\(\PP^d_k\) is
\[
 \Gamma_0=\nu_d(\PP^1_k),\qquad
 \nu_d([s:t])
   =[s^d:s^{d-1}t:\cdots:st^{d-1}:t^d].
\]
After extension of scalars, we use the same notation for the resulting
curve.  A set of points in \(\PP^d\) is in \emph{linearly general
position} if every \(d+1\) of them span \(\PP^d\).

\begin{definition}[\(k\)-forms and split RNCs]
\label{def:rnc-forms}
A closed \(k\)-subscheme
\(\Gamma\subseteq\PP^d_k\) is a \emph{\(k\)-form of an RNC} if
\(\Gamma_{\overline k}\) is projectively equivalent over
\(\overline k\) to \((\Gamma_0)_{\overline k}\).  The form is
\emph{split} if there exists an element of
\(\operatorname{PGL}_{d+1}(k)\) carrying \(\Gamma\) to \(\Gamma_0\).
We use \emph{geometric RNC} for an RNC after base change to an algebraic
closure.
\end{definition}

\begin{lemma}[A rational point splits an RNC form]
\label{lem:rnc-rational-point-split}
Let \(\Gamma\subseteq\PP^d_k\) be a \(k\)-form of an RNC.  If
\(\Gamma(k)\ne\varnothing\), then \(\Gamma\) is split.
\end{lemma}

\begin{proof}
The curve \(\Gamma\) is smooth, projective, geometrically integral, and of
genus zero.  A genus-zero curve with a \(k\)-rational point is
\(k\)-isomorphic to \(\PP^1_k\); choose an isomorphism
\(\psi:\PP^1_k\to\Gamma\).  The line bundle
\[
 \psi^*\mathcal O_\Gamma(1)
\]
has degree \(d\), and hence is isomorphic to
\(\mathcal O_{\PP^1}(d)\).  Moreover, the restriction map
\[
 H^0(\PP^d_k,\mathcal O_{\PP^d}(1))
 \longrightarrow
 H^0(\Gamma,\mathcal O_\Gamma(1))
\]
becomes an isomorphism after extension to \(\overline k\), because the
standard RNC is linearly normal.  It is therefore already an isomorphism
over \(k\).  Thus the given embedding of \(\Gamma\simeq\PP^1_k\) is the
complete linear series \(|\mathcal O_{\PP^1}(d)|\).  Choosing a different
\(k\)-basis of its \(d+1\) global sections amounts to applying an element
of \(\operatorname{PGL}_{d+1}(k)\), so the embedding is
\(k\)-projectively equivalent to \(\nu_d\).  Hence \(\Gamma\) is split.
\end{proof}

\subsection{Uniqueness of the containing curve}

\begin{lemma}[\(d+3\)-point uniqueness]
\label{lem:rnc-uniqueness}
Let \(K\) be an algebraically closed field of arbitrary characteristic,
and let
\[
 P_0,\ldots,P_{d+2}\in\PP^d(K)
\]
be in linearly general position.  If these points lie on an RNC, then
that RNC is unique.
\end{lemma}

This is the characteristic-independent \(d+3\)-point form of
Castelnuovo's lemma; see \cite[Theorem~1.18]{Harris1992} and the
algebraic-closure formulation in \cite[Section~2]{CaminataEtAl2018}.
We apply it only after base change to an algebraic closure.  In particular,
the choice of a projective frame is made over that algebraic closure, and
the lemma imposes no restriction on the size of the ground field.

\subsection{The symmetric-power stabilizer}

Write
\[
 \boldsymbol\nu_d(s,t)
  =(s^d,s^{d-1}t,\ldots,st^{d-1},t^d)^{\mathsf T}.
\]
For \(B\in\operatorname{GL}_2(k)\), the identity
\[
 \boldsymbol\nu_d\!\left(B(s,t)^{\mathsf T}\right)
   =\Sym^d(B)\boldsymbol\nu_d(s,t)
\]
fixes our convention for the symmetric-power representation.  Since
\(\Sym^d(aB)=a^d\Sym^d(B)\), it induces a projective representation
\[
 \overline{\Sym}^d:\operatorname{PGL}_2(k)
       \longrightarrow\operatorname{PGL}_{d+1}(k).
\]
Proposition~\ref{prop:rnc-stabilizer} identifies its image with the
projective stabilizer of \(\Gamma_0\).  We also write
\(\rho_d=\Sym^d\).

\subsection{The orbit-rigidity theorem}

\begin{theorem}[Rational normal curve rigidity]
\label{thm:rnc-rigidity}
Let \(k\) be a field, let \(r\ge3\), let \(n\ge r+3\), and let
\((A,z)\) be a cyclic pair on an \(r\)-dimensional vector space \(V\).
Assume that \(\mathcal K_n(A,z)\) is MDS.  Put \(d=r-1\), and choose
an identification \(V\simeq k^{d+1}\); changing this choice merely
changes the conjugating map below.  Let
\[
 \Gamma_0=\nu_d(\PP^1)\subseteq\PP^d
\]
be the standard rational normal curve.  The following are equivalent.
\begin{enumerate}[label=\textup{(\roman*)}]
\item The orbit points
      \[
       [z],[Az],\ldots,[A^{n-1}z]
      \]
      lie on a \(k\)-form of an RNC in \(\PP(V)\).
\item There exist \(S\in\operatorname{GL}(V)\),
      \(B\in\operatorname{GL}_2(k)\), and \(c\in k^\times\) such that
      \[
       SAS^{-1}=c\,\rho_d(B)
      \]
      and
      \[
       [Sz]\in\Gamma_0(k).
      \]
\end{enumerate}
When these conditions hold, the containing RNC is split, is unique as a
geometric RNC, and is preserved by the projective class of \(A\).
\end{theorem}

\begin{proof}
Assume first that the orbit is contained in a \(k\)-form
\(\Gamma\) of an RNC.  By
Theorem~\ref{thm:krylov-mds}, \(A\) is invertible and every \(r=d+1\)
orbit vectors are linearly independent.  Hence the \(r+2=d+3\) points
\[
 [z],[Az],\ldots,[A^{r+1}z]
\]
are in linearly general position.  Lemma~\ref{lem:rnc-uniqueness},
applied over \(\overline{k}\), shows that \(\Gamma\) is the unique
geometric RNC containing them.

The curves \(\Gamma\) and \(A\Gamma\) both contain
\[
 [Az],[A^2z],\ldots,[A^{r+2}z].
\]
These are again \(d+3\) points in linearly general position; the last
one occurs because \(n\ge r+3\).  A second application of
Lemma~\ref{lem:rnc-uniqueness} after base change gives
\[
 A\Gamma=\Gamma.
\]
Equality after base change implies equality over \(k\) by faithful
flatness.  The form \(\Gamma\) contains the \(k\)-rational point \([z]\),
so Lemma~\ref{lem:rnc-rational-point-split} shows that it is split.
Choose a \(k\)-projectivity carrying \(\Gamma\) to \(\Gamma_0\), and let
\(S\in\operatorname{GL}(V)\) be a lift.  By
Proposition~\ref{prop:rnc-stabilizer}, the projective class of
\(SAS^{-1}\) belongs to
\(\overline{\Sym}^d(\operatorname{PGL}_2(k))\).  Thus there are
\(B\in\operatorname{GL}_2(k)\) and \(c\in k^\times\) such that
\[
 SAS^{-1}=c\,\rho_d(B).
\]
Since \([z]\in\Gamma(k)\), we also have
\([Sz]\in\Gamma_0(k)\).

Conversely, suppose that the displayed relation holds and that
\([Sz]\in\Gamma_0(k)\).  The symmetric-power action preserves
\(\Gamma_0\), and therefore
\[
 [SA^iz]=[c^i\rho_d(B)^iSz]\in\Gamma_0(k)
 \qquad(i\ge0).
\]
Applying \(S^{-1}\) shows that the original orbit is contained in the
RNC \(S^{-1}\Gamma_0\).  The first part of the proof then supplies both
uniqueness and invariance of the containing curve.
\end{proof}

\subsection{Absolute non-RNC consequences}

\begin{corollary}[Absolute rigidity over finite fields]
\label{cor:absolute-rnc}
Let \(k=\F_q\) and retain the hypotheses of
Theorem~\ref{thm:rnc-rigidity}.  If the orbit points are contained in
an RNC after extension to \(\overline{\F}_q\), then their unique geometric
containing RNC is the base change of a split RNC over \(\F_q\).
Consequently, an MDS Krylov code is of GRS type if and only if its orbit
arc is contained in an RNC over \(\overline{\F}_q\).
\end{corollary}

\begin{proof}
The first \(r+2\) orbit points are \(\F_q\)-rational and determine the
geometric containing RNC \(\overline\Gamma\) uniquely.  Let
\[
 I(\overline\Gamma)
 \subseteq
 \overline{\F}_q[X_0,\ldots,X_{r-1}]
\]
be its saturated homogeneous ideal.  The coefficient Frobenius
\(\sigma:a\mapsto a^q\) sends \(\overline\Gamma\) to another geometric
RNC containing the same \(r+2\) rational points.  Uniqueness gives equality
as closed subschemes,
\[
 \sigma(\overline\Gamma)=\overline\Gamma,
 \qquad
 \sigma\bigl(I(\overline\Gamma)\bigr)=I(\overline\Gamma).
\]
For every degree \(e\), the finite-dimensional subspace
\(I(\overline\Gamma)_e\) is therefore stable under the semilinear
Frobenius action.  Galois descent for vector spaces descends these graded
pieces, compatibly with multiplication, to a homogeneous ideal
\[
 I_0\subseteq\F_q[X_0,\ldots,X_{r-1}]
 \quad\text{with}\quad
 I_0\otimes_{\F_q}\overline{\F}_q=I(\overline\Gamma).
\]
Hence
\[
 \Gamma:=\operatorname{Proj}
 \bigl(\F_q[X_0,\ldots,X_{r-1}]/I_0\bigr)
\]
is an \(\F_q\)-form of an RNC whose base change is
\(\overline\Gamma\).  It contains the rational point \([z]\), so
Lemma~\ref{lem:rnc-rational-point-split} makes \(\Gamma\) split.

If the geometric containment holds, the descended split RNC contains all
projective columns of \(H_n(A,z)\); therefore
Proposition~\ref{prop:projective-grs-criterion} shows that the Krylov code
is of GRS type.  Conversely, if the code is of GRS type, that proposition
places its parity-check columns on a split RNC over \(\F_q\), and scalar
extension gives the required geometric containment.
\end{proof}

\begin{corollary}[A companion obstruction to absolute RNC containment]
\label{cor:companion-absolute-nonrnc}
Let \(g\in\F_q[X]\) be monic of degree \(r\ge3\), let \(n\ge r+3\),
and assume that the companion Krylov code
\(\mathcal K_n(T_g,\bar1)\) is MDS.  If either
\begin{enumerate}[label=\textup{(\alph*)}]
\item \(g\) has an irreducible factor of degree at least \(3\), or
\item \(g\) is not squarefree and has at least two distinct roots in
      \(\overline{\F}_q\),
\end{enumerate}
then the companion orbit arc is contained in no RNC over
\(\overline{\F}_q\).  In particular, the code is not of GRS type.
\end{corollary}

\begin{proof}
Suppose that the arc were contained in a geometric RNC.  By
Corollary~\ref{cor:absolute-rnc} and
Theorem~\ref{thm:rnc-rigidity}, there would be
\(S\in\operatorname{GL}_r(\F_q)\), \(B\in\operatorname{GL}_2(\F_q)\),
and \(c\in\F_q^\times\) such that
\[
 ST_gS^{-1}=c\,\Sym^{r-1}(B).
\]
Lemma~\ref{lem:symmetric-power-spectrum} shows that all eigenvalues of
the right-hand side belong to \(\F_{q^2}\).  Since both the minimal
and characteristic polynomials of \(T_g\) equal \(g\), this rules out
an irreducible factor of degree at least \(3\).

If \(g\) has at least two distinct geometric roots, then \(T_g\) has
at least two distinct eigenvalues.  The same spectral lemma makes
\(T_g\) diagonalizable over \(\F_{q^2}\), so its minimal polynomial is
squarefree.  This contradicts condition~\textup{(b)}.
\end{proof}
 \section{Arithmetic classification over finite fields}
\label{sec:finite-field-classification}

Throughout this section,
\[
 r\ge3,\qquad n\ge r+3,\qquad d=r-1,\qquad
 p=\operatorname{char}\F_q,
\]
unless stated otherwise.  For a monic polynomial
\(g\in\F_q[X]\) of degree \(r\), write
\[
 \mathcal C_g:=\mathcal K_n(T_g,\bar1),
\]
and write \(\ord(t)\) for the multiplicative order of a nonzero field
element.  Theorem~\ref{thm:rnc-rigidity} reduces the GRS question to
the three conjugacy types in \(\operatorname{PGL}_2(\F_q)\).  The split
semisimple type gives one family, the nonsplit semisimple type gives
two parity-dependent descent families, and the unipotent type gives
the pure-power family.

\subsection{Split and nonsplit semisimple elements}

The following elementary observation will be used in the sufficiency
part of the classification.

\begin{lemma}[Cyclic points on a semisimple RNC orbit]
\label{lem:cyclic-rnc-point}
Let \(B\in\operatorname{GL}_2(\F_q)\) have distinct eigenvalues
\(\lambda,\mu\in\F_{q^2}\), put \(t=\mu/\lambda\), and assume
\(\ord(t)>d\).  If \(P\in\PP^1(\F_q)\) is not an eigenpoint of \(B\),
then every vector representative of \(\nu_d(P)\) is cyclic for
\(c\,\Sym^d(B)\), for every \(c\in\F_q^\times\).
\end{lemma}

\begin{proof}
Diagonalize \(B\) over \(\F_{q^2}\), and write a representative of
\(P\) as \((s,u)\) in an eigenbasis.  The hypothesis on \(P\) says
\(su\ne0\).  Up to nonzero column factors, the vectors
\[
 \nu_d(P),\ \Sym^d(B)\nu_d(P),\ldots,
 \Sym^d(B)^d\nu_d(P)
\]
form a matrix whose row \(j\) is a nonzero multiple of
\[
 1,t^j,t^{2j},\ldots,t^{dj}
 \qquad(0\le j\le d).
\]
Its determinant is a nonzero Vandermonde product because
\(1,t,\ldots,t^d\) are distinct.  Linear independence after extension
to \(\F_{q^2}\) is equivalent to linear independence over \(\F_q\).
\end{proof}

\subsection{The unipotent family}

\begin{lemma}[Exact MDS criterion for pure powers]
\label{lem:pure-power-mds}
Let \(q\) be a prime power, let \(p=\operatorname{char}\F_q\), let
\(r\ge2\), let \(\beta\in\F_q^\times\), and put
\[
 g(X)=(X-\beta)^r.
\]
Assume that \(n\ge r+2\).  Then
\[
 \mathcal C_g\text{ is MDS}
 \quad\Longleftrightarrow\quad
 n\le p.
\]
Whenever these conditions hold, \(\mathcal C_g\) is of GRS type.
\end{lemma}

\begin{proof}
The root-jet matrix from Theorem~\ref{thm:root-jet-mds} is
\[
 H=\left(\binom{i}{j}\beta^{\,i-j}\right)_{
       0\le j<r,\ 0\le i<n}.
\]
Nonzero row and column scalings transform \(H\) into the Pascal matrix
\[
 P=\left(\binom{i}{j}\right)_{
       0\le j<r,\ 0\le i<n}.
\]
If \(n\le p\), then, for
\(0\le i_1<\cdots<i_r<n\),
\[
 \det\left(\binom{i_b}{j}\right)_{
       \substack{0\le j<r\\1\le b\le r}}
 =
 \frac{\prod_{1\le a<b\le r}(i_b-i_a)}
      {\prod_{j=0}^{r-1}j!}
 \ne0.
\]
Thus the code is MDS.  The polynomials
\(\binom{Y}{0},\ldots,\binom{Y}{r-1}\) form a basis of
\(\F_q[Y]_{<r}\), so the row space of \(P\) is a GRS code evaluated
at the distinct points \(0,1,\ldots,n-1\).  Duality and the preceding
scalings show that \(\mathcal C_g\) is of GRS type.

Suppose now that \(n>p\).  If \(r<p\), the columns of \(P\) indexed by
\(0\) and \(p\) coincide, so the code is not MDS.  Let \(r\ge p\).
If some interior binomial coefficient \(\binom rj\) vanishes modulo
\(p\), then \(g\), viewed as a length-\(n\) codeword, has coefficient
weight at most \(r\), contrary to the MDS criterion.

It remains to consider the case in which every
\(\binom rj\), \(1\le j<r\), is nonzero modulo \(p\).  Lucas' theorem
then forces
\[
 r+1=c p^s
\]
for some \(s\ge1\) and \(1\le c<p\): equivalently, in the base-\(p\)
expansion of \(r\), every digit below the leading nonzero digit is
\(p-1\).  Indeed, any smaller lower digit would allow a number \(j<r\)
whose corresponding digit is larger, and Lucas' product formula would make
\(\binom rj\) vanish.  In characteristic \(p\),
\[
 (X-\beta)^{r+1}
   =(X^{p^s}-\beta^{p^s})^c.
\]
This is a nonzero multiple of \(g\), has degree \(r+1<n\), and has
exactly \(c+1\le p\le r\) nonzero coefficients.  The sparse-multiple
criterion in Theorem~\ref{thm:krylov-mds} again shows that the code is
not MDS.
\end{proof}

\subsection{Complete GRS classification}

\begin{theorem}[Arithmetic classification of the GRS locus]
\label{thm:arithmetic-classification}
Let \(g\in\F_q[X]\) be monic of degree \(r\ge3\), and let
\(n\ge r+3\).  Then \(\mathcal C_g\) is simultaneously MDS and of GRS
type if and only if exactly one of the following mutually exclusive
alternatives holds.
\begin{enumerate}[leftmargin=3.8em]
\item[\textup{(U)}] \textbf{Unipotent case.}
There is a \(\beta\in\F_q^\times\) such that
\[
 g(X)=(X-\beta)^r
 \qquad\text{and}\qquad
 n\le p.
\]

\item[\textup{(S)}] \textbf{Split semisimple case.}
There are \(\gamma,t\in\F_q^\times\) such that
\[
 \ord(t)\ge n,
 \qquad
 g(X)=\prod_{i=0}^{r-1}(X-\gamma t^i).
\]

\item[\textup{(N0)}] \textbf{Nonsplit semisimple case with no linear
factor.}
The degree is even, say \(r=2v\), and there are
\(\gamma,t\in\F_{q^2}^\times\) such that
\[
 t^{q+1}=1,\qquad
 \ord(t)\ge n,\qquad
 \gamma^q=\gamma t^{r-1},
\]
and
\[
 \begin{aligned}
 g(X)
  &=\prod_{i=0}^{r-1}(X-\gamma t^i)\\
  &=\prod_{i=0}^{v-1}
    (X-\gamma t^i)(X-\gamma t^{r-1-i}).
 \end{aligned}
\]
The second line is a product of \(v\) pairwise distinct irreducible
quadratics over \(\F_q\).

\item[\textup{(N1)}] \textbf{Nonsplit semisimple case with one linear
factor.}
The degree is odd, say \(r=2v+1\), and there are
\(\beta\in\F_q^\times\) and \(t\in\F_{q^2}^\times\) such that
\[
 t^{q+1}=1,\qquad
 \ord(t)\ge n,
\]
and
\[
 \begin{aligned}
 g(X)
  &=(X-\beta)\prod_{j=1}^{v}
      (X-\beta t^j)(X-\beta t^{-j})\\
  &=(X-\beta)\prod_{j=1}^{v}
      \bigl(X^2-\beta(t^j+t^{-j})X+\beta^2\bigr).
 \end{aligned}
\]
The displayed quadratic factors are irreducible over \(\F_q\) and
pairwise distinct.
\end{enumerate}
\end{theorem}

\begin{proof}
Let \(T_g\) be multiplication by \(X\) on \(\F_q[X]/(g)\).  It is a
cyclic operator whose minimal and characteristic polynomials are both
\(g\).

\smallskip
\noindent\emph{Necessity: the three projective types.}
Assume that \(\mathcal C_g\) is MDS and of GRS type.  Applying
Theorem~\ref{thm:rnc-rigidity} to the cyclic pair
\((T_g,\bar1)\) gives
\[
 ST_gS^{-1}=c\,\Sym^d(B)
\]
for some \(S\in\operatorname{GL}_r(\F_q)\),
\(c\in\F_q^\times\), and \(B\in\operatorname{GL}_2(\F_q)\).
We distinguish the conjugacy type of \([B]\).

Suppose first that \(B\) is semisimple, with distinct eigenvalues
\(\lambda,\mu\in\F_{q^2}^\times\).  In the nonsplit case this
diagonalization is over \(\F_{q^2}\), and no diagonalization over
\(\F_q\) is being asserted.  Put
\[
 \gamma=c\lambda^d,\qquad t=\frac{\mu}{\lambda}.
\]
Lemma~\ref{lem:symmetric-power-spectrum} shows that the roots of \(g\)
are
\[
 \gamma,\gamma t,\ldots,\gamma t^d.
\]
They are distinct because a diagonalizable cyclic operator has simple
spectrum.

We claim that \(\ord(t)\ge n\).  If
\(m=\ord(t)<n\), then every eigenvalue of \(T_g^m\) equals
\(\gamma^m\).  Since \(T_g\) is diagonalizable over \(\F_{q^2}\),
\[
 T_g^m=\gamma^m I_r.
\]
Consequently \([T_g^m\bar1]=[\bar1]\), so two projective columns of
the MDS orbit arc coincide, a contradiction.

If \(B\) is split semisimple, then
\(\lambda,\mu\in\F_q^\times\), and hence
\(\gamma,t\in\F_q^\times\).  This is family \textup{(S)}.

Now suppose that \(B\) is nonsplit semisimple.  After interchanging
its eigenvalues, take \(\mu=\lambda^q\).  Then
\[
 t^q=t^{-1},\qquad t^{q+1}=1,\qquad
 \gamma^q=\gamma t^d.
\]
Frobenius acts on the progression by reversal:
\[
 (\gamma t^i)^q=\gamma t^{d-i}.
\]
Because \(\ord(t)\ge n>d\), the root \(\gamma t^i\) is fixed by
Frobenius exactly when \(2i=d\) as an equality of integers.
If \(d\) is odd, equivalently \(r=2v\) is even, there is no fixed root;
the pairs \(i\leftrightarrow d-i\), \(0\le i<v\), give the
irreducible quadratic factorization in \textup{(N0)}.
If \(d=2v\) is even, equivalently \(r=2v+1\) is odd, there is exactly
one fixed root.  Put
\[
 \beta=\gamma t^v\in\F_q^\times.
\]
The root set becomes
\[
 \{\beta t^{-v},\ldots,\beta t^{-1},\beta,
   \beta t,\ldots,\beta t^v\},
\]
which gives the factorization in \textup{(N1)}.
The order inequality shows that no noncentral paired root is
\(\F_q\)-rational and that two different pairs cannot coincide.
Thus all the displayed quadratics are irreducible and pairwise
distinct.

Finally suppose that \(B\) has only one eigenvalue.  The scalar case
would make \(T_g\) scalar, contradicting cyclicity for \(r\ge3\).
Thus \([B]\) is unipotent.  By
Lemma~\ref{lem:symmetric-power-spectrum}, \(T_g\) has a single
geometric eigenvalue \(\beta\), and therefore
\[
 g(X)=(X-\beta)^r.
\]
The unique root is Frobenius-fixed and \(T_g\) is invertible, so
\(\beta\in\F_q^\times\).  Lemma~\ref{lem:pure-power-mds} gives
\(n\le p\), proving \textup{(U)}.

\smallskip
\noindent\emph{Sufficiency.}
Family \textup{(U)} is MDS and GRS by
Lemma~\ref{lem:pure-power-mds}.  In each semisimple family the roots
form a progression with ratio of order at least \(n\), so
Lemma~\ref{lem:semisimple-orbit-mds} proves the MDS property.

For \textup{(S)}, take
\[
 B=\begin{pmatrix}1&0\\0&t\end{pmatrix},
 \qquad A=\gamma\,\Sym^d(B).
\]
Then \(A\) is defined over \(\F_q\), preserves \(\Gamma_0\), and has
characteristic polynomial \(g\).

For either nonsplit family, first write its root set as
\[
 \{\gamma,\gamma t,\ldots,\gamma t^d\},
 \qquad
 \gamma^q=\gamma t^d;
\]
in family \textup{(N1)} take \(\gamma=\beta t^{-v}\).
The map \(x\mapsto x^{q-1}\) from \(\F_{q^2}^\times\) onto its
norm-one subgroup is surjective, so choose
\(\lambda\in\F_{q^2}^\times\) with
\[
 \lambda^{q-1}=t.
\]
Let \(B\in\operatorname{GL}_2(\F_q)\) represent multiplication by
\(\lambda\) on the two-dimensional \(\F_q\)-space \(\F_{q^2}\).
Its eigenvalues over \(\F_{q^2}\) are
\(\lambda\) and \(\lambda^q=\lambda t\).  With
\[
 c=\frac{\gamma}{\lambda^d},
\]
the descent equation gives \(c^q=c\), so \(c\in\F_q^\times\), and
\[
 A=c\,\Sym^d(B)
\]
has characteristic polynomial \(g\) and preserves \(\Gamma_0\).
This construction uses diagonalization only over \(\F_{q^2}\).

In every semisimple case, the order condition makes the eigenvalues
of \(A\) distinct.  Thus \(A\) is cyclic and is similar over
\(\F_q\) to \(T_g\), since both have characteristic polynomial \(g\).
Choose \(R\in\operatorname{GL}_r(\F_q)\) with
\[
 RT_gR^{-1}=A,
 \qquad w=R\bar1.
\]
Then
\[
 H_n(A,w)=R H_n(T_g,\bar1),
 \qquad
 \mathcal K_n(T_g,\bar1)=\mathcal K_n(A,w).
\]
Choose \(P\in\PP^1(\F_q)\) away from the two eigenpoints in the split
case; in the nonsplit case every \(\F_q\)-rational point has this
property.  Lemma~\ref{lem:cyclic-rnc-point} makes
\(z=\nu_d(P)\) cyclic for \(A\).  The vector \(w\) is cyclic as well,
and Theorem~\ref{thm:intrinsic-krylov} therefore gives
\[
 \mathcal K_n(A,w)=\mathcal K_n(A,z).
\]
Consequently \(\mathcal C_g=\mathcal K_n(A,z)\).  Its projective
parity-check orbit lies on
\(\Gamma_0\), so it is of GRS type.

The alternatives are mutually exclusive.  Family \textup{(U)} is the
only non-squarefree family; among the semisimple families,
\textup{(S)} has all roots in \(\F_q\), \textup{(N0)} has none, and
\textup{(N1)} has exactly one.
\end{proof}

\subsection{The MDS/GRS/non-GRS trichotomy}

For fixed \(q,n,r\) in the stable range, let
\(\mathcal P_{n,r}(q)\) be the set of monic degree-\(r\) polynomials
over \(\F_q\).  Let
\(\mathcal U,\mathcal S,\mathcal N_0,\mathcal N_1\) denote the four
sets in Theorem~\ref{thm:arithmetic-classification}, and put
\[
 \begin{aligned}
 \mathcal M_{n,r}(q)
  &=\{g\in\mathcal P_{n,r}(q):D_{n,r}(g)\ne0\},\\
 \mathcal G_{n,r}(q)
  &=\mathcal U\mathbin{\dot\cup}\mathcal S
    \mathbin{\dot\cup}\mathcal N_0\mathbin{\dot\cup}\mathcal N_1,\\
 \mathcal N_{n,r}(q)
  &=\mathcal M_{n,r}(q)\setminus\mathcal G_{n,r}(q).
 \end{aligned}
\]

\begin{theorem}[Complete MDS/GRS trichotomy]
\label{thm:complete-trichotomy}
Let \(g\in\mathcal P_{n,r}(q)\).  Exactly one of the following
alternatives holds.
\begin{enumerate}[leftmargin=3.2em]
\item[\textup{(A)}] \textbf{Non-MDS:}
      \(D_{n,r}(g)=0\).
\item[\textup{(B)}] \textbf{MDS-GRS:}
      \(g\in\mathcal G_{n,r}(q)\), equivalently \(g\) belongs to
      exactly one of
      \(\mathcal U,\mathcal S,\mathcal N_0,\mathcal N_1\).
\item[\textup{(C)}] \textbf{MDS-non-GRS:}
      \(D_{n,r}(g)\ne0\) and
      \(g\notin\mathcal G_{n,r}(q)\).
\end{enumerate}
Thus
\[
 \mathcal P_{n,r}(q)
 =
 \bigl(\mathcal P_{n,r}(q)\setminus\mathcal M_{n,r}(q)\bigr)
 \mathbin{\dot\cup}\mathcal G_{n,r}(q)
 \mathbin{\dot\cup}\mathcal N_{n,r}(q).
\]
\end{theorem}

\begin{proof}
Theorem~\ref{thm:determinantal-mds} identifies the MDS locus with
\(D_{n,r}\ne0\).  Lemmas~\ref{lem:pure-power-mds} and
\ref{lem:semisimple-orbit-mds} put all four arithmetic families inside
that locus, and Theorem~\ref{thm:arithmetic-classification} identifies
their union with the complete GRS sublocus.  Its complement inside the
MDS locus is precisely the MDS non-GRS locus.
\end{proof}

\begin{corollary}[Factors of degree at least three]
\label{cor:high-degree-factor-nongrs}
Let \(g\in\F_q[X]\) be monic of degree \(r\ge3\), let \(n\ge r+3\),
and assume that \(\mathcal C_g\) is MDS.  If \(g\) has an irreducible
factor of degree at least \(3\), then \(\mathcal C_g\) is not of GRS
type.
\end{corollary}

\begin{proof}
Every polynomial in the four GRS families factors over \(\F_q\) into
linear and quadratic factors.  Apply
Theorem~\ref{thm:arithmetic-classification}.
\end{proof}

\begin{corollary}[Genuine repeated-root polynomials]
\label{cor:repeated-nonpure-nongrs}
Let \(g\in\F_q[X]\) be monic of degree \(r\ge3\), let \(n\ge r+3\),
and assume that \(\mathcal C_g\) is MDS.  If \(g\) is not squarefree
and is not a pure power of a nonzero linear polynomial, then
\(\mathcal C_g\) is not of GRS type.
\end{corollary}

\begin{proof}
The semisimple families in
Theorem~\ref{thm:arithmetic-classification} are squarefree, and its
only non-squarefree family is \textup{(U)}.
\end{proof}
 \section{The coefficient surface of the GRS locus}
\label{sec:coefficient-surface}

Throughout this section, \(r\geq 3\) and \(d=r-1\).  We pass from the
orbit-theoretic classification to the coefficient space of monic
degree-\(r\) polynomials.  The semisimple part of the GRS locus is governed
by root sets in geometric progression.  Its closure is only
two-dimensional, independently of \(r\), and its natural parameterization
has precisely the reversal ambiguity that also controls arithmetic descent.

\subsection{The geometric-progression parameterization}

\begin{definition}[The coefficient surface]
For every field \(k\), let
\[
\Phi_{r,k}:\mathbb G_{m,k}^{2}\longrightarrow\A_k^r
\]
be the coefficient morphism defined by
\[
\prod_{i=0}^{d}(X-\gamma t^i)
=X^r+a_{r-1}X^{r-1}+\cdots+a_0,
\qquad
\Phi_{r,k}(\gamma,t)=(a_0,\ldots,a_{r-1}).
\]
Put
\[
\GRSsurf_{r,k}
:=\overline{\operatorname{im}\Phi_{r,k}}
\subseteq\A_k^r,
\]
where this is the Zariski closure of the underlying topological image,
endowed with its reduced induced subscheme structure,
and call it the \emph{coefficient surface}.  When the ground field is clear,
we omit the subscript \(k\).  On the parameter torus, define the reversal
involution
\[
\iota(\gamma,t)=(\gamma t^d,t^{-1}).
\]
\end{definition}

The involution reverses the order of a geometric progression without
changing its underlying root set.  The next result shows that, generically,
this is the only ambiguity.

\subsection{Reversal and the generic degree-two quotient}

\begin{theorem}[Geometry of the coefficient surface]
\label{thm:coefficient-surface}
Let \(k\) be an algebraically closed field and let \(r\geq3\).  Then the
following statements hold.
\begin{enumerate}[label=\textup{(\arabic*)}]
\item \(\Phi_{r,k}\circ\iota=\Phi_{r,k}\).
\item \(\GRSsurf_{r,k}\) is an irreducible rational surface.
\item The generic fiber of \(\Phi_{r,k}\) is
      \[
      \{(\gamma,t),(\gamma t^d,t^{-1})\}.
      \]
      In particular, \(\Phi_{r,k}\) is generically finite of degree two.
\item
      \[
      \dim\GRSsurf_{r,k}=2,
      \qquad
      \operatorname{codim}_{\A_k^r}\GRSsurf_{r,k}=r-2.
      \]
\item The morphism \(\Phi_{r,k}\) factors through \(D(a_0)\), and the
      induced morphism
      \[
      \Phi_{r,k}:\mathbb G_{m,k}^2\longrightarrow D(a_0)
      \]
      is finite.  If \(\operatorname{Im}_{\mathrm{sch}}(\Phi_{r,k})\)
      denotes its scheme-theoretic image, then, as reduced closed
      subschemes of \(D(a_0)\),
      \[
      \GRSsurf_{r,k}\cap D(a_0)
      =\operatorname{Im}_{\mathrm{sch}}(\Phi_{r,k}).
      \]
      In particular, when \(k\) is algebraically closed,
      \[
      \bigl(\GRSsurf_{r,k}\cap D(a_0)\bigr)(k)
      =\Phi_{r,k}\bigl(\mathbb G_{m,k}^2(k)\bigr).
      \]
\end{enumerate}
\end{theorem}

\begin{proof}
Reversing the roots gives
\[
\{\gamma,\gamma t,\ldots,\gamma t^d\}
=
\{\gamma t^d,\gamma t^{d-1},\ldots,\gamma\},
\]
so \(\Phi_{r,k}\circ\iota=\Phi_{r,k}\).

Consider the monomial homomorphism
\[
\psi_k:\mathbb G_{m,k}^{2}\longrightarrow(\mathbb G_{m,k})^r,
\qquad
(\gamma,t)\longmapsto
(\gamma,\gamma t,\ldots,\gamma t^d).
\]
On coordinate rings, \(\psi_k\) is induced by
\[
\psi_k^\#:
k[x_0^{\pm1},\ldots,x_d^{\pm1}]
\longrightarrow k[\gamma^{\pm1},t^{\pm1}],
\qquad
x_i\longmapsto\gamma t^i.
\]
This homomorphism is surjective, since
\[
\gamma=\psi_k^\#(x_0),
\qquad
t=\psi_k^\#(x_1x_0^{-1}).
\]
Hence \(\psi_k\) is a closed immersion.

The torus \(\Gm^2\) is irreducible, and hence so is the closure of its
image.  Put
\[
K=k(\gamma,t),
\qquad
L=k(\GRSsurf_{r,k}),
\]
where pullback along \(\Phi_{r,k}\) identifies \(L\) with a subfield of
\(K\).  The factorization through the finite symmetric-group quotient
described below shows that \(K/L\) is finite.  To compute its generic
fiber, regard \(\gamma\) and \(t\) as algebraically independent over
\(k\).  Suppose that
\[
\{\gamma t^i:0\leq i\leq d\}
=
\{\delta s^i:0\leq i\leq d\}.
\]
The element \(s\) is a quotient of two roots on the left, so
\(s=t^e\) for some \(e\in\mathbb Z\); similarly,
\(\delta=\gamma t^a\) for some \(a\in\mathbb Z\).  Since \(t\) is
transcendental, equality of the root sets is equivalent to
\[
\{a,a+e,\ldots,a+de\}=\{0,1,\ldots,d\}
\]
inside \(\mathbb Z\).  The roots are distinct, so \(e\neq0\).  If
\(e>0\), comparison of the minimum and maximum
gives \(a=0\) and \(e=1\).  If \(e<0\), it gives \(a=d\) and \(e=-1\).
Thus the generic fiber consists precisely of
\((\gamma,t)\) and \((\gamma t^d,t^{-1})\).  At the generic point the
\(r\) roots are distinct.  Let
\(U_{\mathrm{dist}}\subset(\mathbb G_{m,k})^r\) be the distinct-root open
subset.  The constant finite \'etale group scheme \(\underline{S_r}_k\) acts
freely on \(U_{\mathrm{dist}}\).  Hence
\[
U_{\mathrm{dist}}
\longrightarrow U_{\mathrm{dist}}/\underline{S_r}_k
\]
is an \(\underline{S_r}_k\)-torsor and therefore finite \'etale in every
characteristic.
Since \(\psi_k\) is a closed immersion, the differential of the composite
has rank two at the generic point.  Thus the induced finite extension
\(K/L\) is separable.  The two geometric generic points just found therefore
give
\[
[K:L]=2.
\]
In particular, \(\dim\GRSsurf_{r,k}=2\), and the stated codimension follows.

The identity \(\Phi_{r,k}\circ\iota=\Phi_{r,k}\) gives
\(L\subseteq K^{\langle\iota\rangle}\).  The automorphism \(\iota\) is
nontrivial because it sends the transcendental element \(t\) to
\(t^{-1}\).  Artin's fixed-field theorem, which applies in every
characteristic (including characteristic two), gives
\[
[K:K^{\langle\iota\rangle}]=2.
\]
The tower law now yields \(L=K^{\langle\iota\rangle}\).

It remains to prove rationality.  In \(K\), put
\[
u=t+t^{-1},
\qquad
v=\gamma(1+t^d).
\]
Both functions are invariant under \(\iota\).  Conversely, \(t\) satisfies
\[
T^2-uT+1=0,
\]
and
\[
\gamma=\frac{v}{1+t^d}.
\]
The denominator is nonzero because \(t\) is transcendental, also in
characteristic two.  Hence \([K:k(u,v)]\leq2\).  Since
\(k(u,v)\subseteq K^{\langle\iota\rangle}\) and the fixed field has
index two, equality holds:
\[
K^{\langle\iota\rangle}=k(u,v).
\]
The finite extension \(K/k(u,v)\) preserves transcendence degree, so
\(u\) and \(v\) are algebraically independent.  Consequently,
\[
k(\GRSsurf_{r,k})=L=K^{\langle\iota\rangle}=k(u,v),
\]
which proves that \(\GRSsurf_{r,k}\) is rational.

For the final assertion, use the closed immersion \(\psi_k\) above.  The
quotient morphism
\[
(\mathbb G_{m,k})^r
\longrightarrow\operatorname{Sym}^r(\mathbb G_{m,k})
\]
is finite, since it is the quotient by the finite symmetric group.  Under
the coefficient identification
\[
\operatorname{Sym}^r(\mathbb G_{m,k})
\simeq D(a_0)
=\operatorname{Spec}k[a_0^{\pm1},a_1,\ldots,a_{r-1}],
\]
the composite with \(\psi_k\) is exactly the factorization of
\(\Phi_{r,k}\) through \(D(a_0)\).  Since a closed immersion is finite, this
composite is finite.  Put
\(Z=\operatorname{Im}_{\mathrm{sch}}(\Phi_{r,k})\).  On coordinate rings,
\[
\mathcal O(Z)
=k[a_0^{\pm1},a_1,\ldots,a_{r-1}]/\ker(\Phi_{r,k}^{\#})
\hookrightarrow k[\gamma^{\pm1},t^{\pm1}].
\]
Thus \(Z\) is reduced, indeed integral.  Since a finite morphism is closed,
\(\lvert Z\rvert\) is the closed set-theoretic image of \(\Phi_{r,k}\) in
\(D(a_0)\).  By the definition of \(\GRSsurf_{r,k}\), the reduced closed
subschemes \(Z\) and \(\GRSsurf_{r,k}\cap D(a_0)\) have the same underlying
closed subset and hence are equal.  If \(k\) is algebraically closed, the
finite surjection \(\mathbb G_{m,k}^2\to Z\) is surjective on closed
\(k\)-points, which proves the final equality.
\end{proof}

\subsection{Normalization and the fixed locus}
\label{subsec:normalization-fixed-locus}

\begin{theorem}[Normalization over the nonzero-constant-term locus]
\label{thm:normalization-open}
Let \(k\) be an algebraically closed field, let \(r\geq3\), and put
\(d=r-1\).  Set
\[
T=\mathbb G_{m,k}^{2},
\qquad
A=k[\gamma^{\pm1},t^{\pm1}],
\qquad
Y_r=\GRSsurf_{r,k}\cap D(a_0).
\]
Let \(\iota^*:A\to A\) be the involution determined by
\[
\iota^*(\gamma)=\gamma t^d,
\qquad
\iota^*(t)=t^{-1},
\]
and define the affine categorical quotient by
\[
T/\!/\langle\iota\rangle
:=\operatorname{Spec}C,
\qquad
C=A^{\langle\iota^*\rangle}.
\]
Let \(q:T\to T/\!/\langle\iota\rangle\) be the canonical quotient
morphism. Then \(q\) is finite, and the coefficient morphism
\(\Phi_{r,k}:T\to Y_r\) factors uniquely as
\[
T
\xrightarrow{\;q\;}
T/\!/\langle\iota\rangle
\xrightarrow{\;\overline{\Phi}_{r,k}\;}
Y_r,
\]
and \(\overline{\Phi}_{r,k}\) is the normalization morphism of \(Y_r\).
Equivalently,
\[
Y_r^\nu\simeq T/\!/\langle\iota\rangle.
\]
This assertion holds in every characteristic, including characteristic
two.
\end{theorem}

\begin{proof}
For every \(f\in A\),
\[
\bigl(Z-f\bigr)\bigl(Z-\iota^*(f)\bigr)
=Z^2-\bigl(f+\iota^*(f)\bigr)Z+f\iota^*(f)
\in C[Z],
\]
because both coefficients are fixed by \(\iota^*\). Hence every element of
\(A\) is integral over \(C\). Moreover, since \(k\subseteq C\),
\[
A=C[\gamma,\gamma^{-1},t,t^{-1}],
\]
so \(A\) is a finitely generated \(C\)-algebra. Since \(A\) is both integral
and of finite type over \(C\), it is finite as a \(C\)-module. Therefore the
quotient morphism \(q\) is finite. This argument is valid in every
characteristic.

Let \(B=\mathcal O(Y_r)\).  By
Theorem~\ref{thm:coefficient-surface}, the morphism
\(\Phi_{r,k}:T\to Y_r\) is finite and has scheme-theoretic image \(Y_r\).
Consequently, \(B\) is a finitely generated \(k\)-domain,
\(\Phi_{r,k}^{\#}\) identifies \(B\) with a subring of \(A\), and \(A\)
is finite as a \(B\)-module.

The morphism \(\Phi_{r,k}\) is invariant under \(\iota\).  Hence
\[
B\subseteq C=A^{\langle\iota^*\rangle}\subseteq A,
\]
which gives the asserted factorization.  Since \(B\) is Noetherian and
\(C\) is a \(B\)-submodule of the finite \(B\)-module \(A\), the ring
\(C\) is finite over \(B\).  Thus \(\overline{\Phi}_{r,k}\) is finite.

We next verify that \(C\) is normal without restricting the characteristic.
The Laurent polynomial ring \(A\) is a normal domain.  Let
\(x\in\operatorname{Frac}(C)\) be integral over \(C\).  The same monic
equation, now viewed as an equation with coefficients in \(A\), shows that
\(x\) is integral over \(A\).  Since \(x\in\operatorname{Frac}(A)\) and
\(A\) is normal, we obtain \(x\in A\).  Moreover, every element of
\(\operatorname{Frac}(C)\) is fixed by \(\iota^*\), so
\(x\in A^{\langle\iota^*\rangle}=C\).  Therefore \(C\) is normal.  Notice
that this argument uses neither averaging nor the invertibility of two.

Let \(K=\operatorname{Frac}(A)\).  Since \(Y_r\) is a dense open subscheme
of \(\GRSsurf_{r,k}\), Theorem~\ref{thm:coefficient-surface} gives
\[
\operatorname{Frac}(B)
=k(Y_r)
=k(\GRSsurf_{r,k})
=K^{\langle\iota^*\rangle}.
\]
The inclusions \(B\subseteq C\subseteq A\) therefore imply
\[
\operatorname{Frac}(B)
\subseteq\operatorname{Frac}(C)
\subseteq K^{\langle\iota^*\rangle}
=\operatorname{Frac}(B).
\]
Thus \(\operatorname{Frac}(C)=\operatorname{Frac}(B)\), and
\(\overline{\Phi}_{r,k}\) is finite and birational.  Since \(C\) is
integral over \(B\), it is contained in the integral closure of \(B\) in
their common fraction field. Conversely, if \(x\) in this common fraction
field is integral over \(B\), then the same monic polynomial shows that
\(x\) is integral over \(C\). Normality of \(C\) therefore gives \(x\in C\).
Hence \(C\) is exactly the integral closure of \(B\), and
\(\overline{\Phi}_{r,k}\) is the normalization morphism.
\end{proof}

Thus the generic reversal ambiguity is promoted from a function-field
statement to a global normalization statement on the
nonzero-constant-term locus.

\begin{corollary}[Fixed locus of reversal and the free quotient locus]
\label{cor:reversal-fixed-locus}
In the notation of Theorem~\ref{thm:normalization-open}, assume that
\(\operatorname{char}k\neq2\), and let
\(q:T\to T/\!/\langle\iota\rangle\) be the quotient morphism.  The fixed
locus of \(\iota\) is
\[
F=
\begin{cases}
V(t-1),&r\text{ even},\\
V(t-1)\,\dot\cup\,V(t+1),&r\text{ odd}.
\end{cases}
\]
Since \(q\) is finite, its image \(q(F)\) is closed. Put
\(U=T\setminus F\) and
\(V=(T/\!/\langle\iota\rangle)\setminus q(F)\).  Then
\(U=q^{-1}(V)\), and
\[
q|_U:U\longrightarrow V
\]
is finite \(\acute{e}\)tale of degree two.
\end{corollary}

\begin{proof}
We use the standard facts that, for a finite group acting on an affine
scheme, the primes over a fixed prime of the invariant ring form a single
group orbit, and that on an invariant free affine open the quotient is a
torsor under the corresponding constant finite \(\acute{e}\)tale group
scheme; see
\cite[Tags~\texttt{0BRI} and~\texttt{07S7}]{StacksProject}.

A point is fixed precisely when
\(t=t^{-1}\) and \(\gamma=\gamma t^d\), or equivalently when
\[
t^2=1,
\qquad
t^d=1.
\]
Because \(1\) and \(-1\) are distinct, these equations give the stated
fixed locus according as \(d=r-1\) is odd or even.  The orbit-fiber fact
gives \(q^{-1}(q(F))=F\): an orbit meeting the fixed locus is the singleton
orbit of that fixed point. Hence \(U=q^{-1}(V)\). Moreover, \(U\) is the
invariant principal affine open
\[
U=
\begin{cases}
D(t-1),&r\text{ even},\\
D((t-1)(t+1)),&r\text{ odd}.
\end{cases}
\]
The fixed equalizer is empty on \(U\), so the constant group scheme
\(\underline{\mathbb Z/2\mathbb Z}_k\) acts scheme-theoretically freely
there. The restricted quotient is therefore a torsor under this finite
\(\acute{e}\)tale group scheme of rank two, and hence is finite
\(\acute{e}\)tale of degree two.
\end{proof}

\begin{remark}[The fixed subscheme in characteristic two]
Assume that \(\operatorname{char}k=2\), and write \(d=r-1\). The equalizer
of \(\iota\) and the identity is defined in
\(A=k[\gamma^{\pm1},t^{\pm1}]\) by
\[
I_{\operatorname{Fix}(\iota)}
=\bigl(\iota^*(t)-t,\iota^*(\gamma)-\gamma\bigr)
=\bigl(t^{-1}-t,\gamma(t^d-1)\bigr)
=\bigl((t-1)^2,t^d-1\bigr).
\]
If \(d\) is odd, write \(t^d-1=(t-1)h_d(t)\). Since
\(h_d(1)=d=1\) in \(k\), we have \((h_d,t-1)=A\), and hence
\(I_{\operatorname{Fix}(\iota)}=(t-1)\). If \(d\) is even, say
\(d=2m\), then
\[
t^d-1=(t^m-1)^2\in\bigl((t-1)^2\bigr),
\]
so \(I_{\operatorname{Fix}(\iota)}=((t-1)^2)\). Consequently, the fixed
subscheme is \(V(t-1)\) when \(r\) is even and \(V((t-1)^2)\) when \(r\)
is odd; in either case its underlying closed subset is \(V(t-1)\).

Let \(U_2=D(t-1)\) and
\[
V_2=(T/\!/\langle\iota\rangle)\setminus q(V(t-1)).
\]
The orbit-fiber property gives \(U_2=q^{-1}(V_2)\). On the invariant
principal affine open \(U_2\), the fixed subscheme is empty, so the constant
group scheme \(\underline{\mathbb Z/2\mathbb Z}_k\) acts
scheme-theoretically freely. The same quotient-torsor argument shows that
\(q|_{U_2}:U_2\to V_2\) is finite \(\acute{e}\)tale of degree two.

The normalization theorem concerns only
\(\GRSsurf_{r,k}\cap D(a_0)\).  It neither identifies the normalization
across the boundary \(a_0=0\) nor describes the full defining ideal,
conductor, or local ramification structure of the quotient and normalization
morphisms.
\end{remark}

\subsection{The squarefree GRS locus and the pure-power boundary}

Let \(k\) be a field.  All coefficient-space loci in this subsection are
formed in \(\A_k^r\).  In particular, \(\MDSopen_{n,r}\) and
\(D(\Disc)\) denote the base changes to \(k\) of the corresponding principal
open subsets defined over \(\mathbb Z\).

For a coefficient point
\(a=(a_0,\ldots,a_{r-1})\in\A_k^r\), write
\[
g_a(X)=X^r+a_{r-1}X^{r-1}+\cdots+a_0.
\]
Write
\[
\Disc=\Disc_r(a_0,\ldots,a_{r-1})
      :=(-1)^{r(r-1)/2}\operatorname{Res}_X(g_a,g_a')
\]
for the discriminant polynomial on coefficient space.  Thus
\(D(\Disc)\) is the squarefree open subset.
By Theorem~\ref{thm:determinantal-mds}, the companion code \(\mathcal C_{g_a}\)
is MDS precisely on the principal open subset
\[
\MDSopen_{n,r}:=D(D_{n,r})\subseteq\A_k^r.
\]
Let
\[
 u_{r,k}:\A^1_k\longrightarrow\A^r_k,
 \qquad
 \beta\longmapsto\coeff\bigl((X-\beta)^r\bigr),
\]
and let \(\Upower_{r,k}^{\mathrm{cl}}\subseteq\A^r_k\) be the Zariski
closure of its underlying topological image, endowed with the reduced induced
structure.  Define the nonzero-constant-term part by
\[
 \Upower_{r,k}^{\times}:=\Upower_{r,k}^{\mathrm{cl}}\cap D(a_0).
\]
When \(k\) is algebraically closed, these sets of \(k\)-points are
\[
\Upower_{r,k}^{\mathrm{cl}}(k)
=\left\{\coeff\bigl((X-\beta)^r\bigr):\beta\in k\right\},
\qquad
\Upower_{r,k}^{\times}(k)
=\left\{\coeff\bigl((X-\beta)^r\bigr):\beta\in k^\times\right\}.
\]
Thus \(\Upower_{r,k}^{\mathrm{cl}}\) contains the limiting point corresponding
to \(X^r\), whereas \(\Upower_{r,k}^{\times}\) is the part relevant to invertible
companion operators.
For algebraically closed \(k\), let
\(\operatorname{GRS}_{n,r}(k)\) denote the coefficient points for which
the companion Krylov code is MDS and its projective orbit is contained in a
rational normal curve.

\begin{theorem}[The geometric GRS locus inside coefficient space]
\label{thm:geometric-grs-locus}
Let \(k\) be an algebraically closed field, let \(r\geq3\), and let
\(n\geq r+3\).  Then
\[
\boxed{
\operatorname{GRS}_{n,r}(k)
=
\bigl(
\MDSopen_{n,r}\cap\GRSsurf_{r,k}\cap D(\Disc)
\bigr)(k)
\;\dot\cup\;
\bigl(
\MDSopen_{n,r}\cap\Upower_{r,k}^{\times}
\bigr)(k).
}
\]
Moreover,
\[
\bigl(\MDSopen_{n,r}\cap\Upower_{r,k}^{\times}\bigr)(k)
=
\begin{cases}
\Upower_{r,k}^{\times}(k),
   & \operatorname{char}k=0,\\
\Upower_{r,k}^{\times}(k),
   & \operatorname{char}k=p>0\text{ and }n\leq p,\\
\varnothing,
   & \operatorname{char}k=p>0\text{ and }n>p.
\end{cases}
\]
\end{theorem}

\begin{proof}
Let
\(a\in(\MDSopen_{n,r}\cap\GRSsurf_{r,k}\cap D(\Disc))(k)\).  The MDS
condition forces \(a_0\neq0\), since the companion operator must be
invertible.  Theorem~\ref{thm:coefficient-surface}\textup{(5)} therefore
supplies \(\gamma,t\in k^\times\) such that
\[
g_a(X)=\prod_{i=0}^{d}(X-\gamma t^i).
\]
The discriminant condition makes these roots distinct.  If \(t^m=1\) for
some \(1\leq m<n\), then in the root-evaluation parity-check matrix the
column indexed by \(j+m\) is \(\gamma^m\) times the column indexed by \(j\)
whenever both occur.  This contradicts the MDS property.  Thus
\(\ord(t)\geq n\), where infinite order is allowed.  Taking
\[
B=\begin{pmatrix}1&0\\0&t\end{pmatrix},
\qquad A=\gamma\,\operatorname{Sym}^d(B),
\]
the eigenvalues of \(A\) are exactly the roots of \(g_a\).  They are
distinct, so \(A\) is cyclic and is similar to \(T_{g_a}\).  A point of
\(\mathbb P^1\) with both eigen-coordinates nonzero is cyclic for \(A\) by
the same Vandermonde calculation as in
Lemma~\ref{lem:cyclic-rnc-point}.  Its orbit lies on the standard RNC, and
Theorem~\ref{thm:intrinsic-krylov} transfers that orbit description to the
companion pair.  Lemma~\ref{lem:semisimple-orbit-mds} verifies the same
order condition at the level of maximal minors.

Now suppose \(g_a=(X-\beta)^r\).  Its Hasse-jet parity-check matrix is,
up to nonzero row and column scalings, the Pascal matrix
\[
\left(\binom{i}{j}\right)_{
0\leq j<r,\ 0\leq i<n}.
\]
In characteristic zero, for column indices
\(0\leq i_1<\cdots<i_r<n\), its maximal minor is
\[
\det\!\left(\binom{i_b}{j}\right)_
 {0\leq j<r,\,1\leq b\leq r}
=
\frac{\prod_{1\leq a<b\leq r}(i_b-i_a)}
     {\prod_{j=0}^{r-1}j!},
\]
which is nonzero.  In characteristic \(p>0\), every maximal minor of this
Pascal matrix lies in the prime field \(\F_p\), and its vanishing is unchanged
after scalar extension from \(\F_p\) to \(k\).  Applying
Lemma~\ref{lem:pure-power-mds} over \(\F_p\), with \(\beta=1\), therefore
shows that all maximal minors are nonzero precisely when \(n\leq p\).  In the
MDS cases the Pascal columns, after the indicated row change, are distinct
points of the standard rational normal curve.
This proves the inclusion from the right-hand side to the left and the
displayed equality of \(k\)-point sets for
\(\MDSopen_{n,r}\cap\Upower_{r,k}^{\times}\).

Conversely, assume that the companion orbit is MDS and lies on a rational
normal curve.  By Theorem~\ref{thm:rnc-rigidity}, its companion operator is
conjugate, up to a nonzero scalar, to \(\operatorname{Sym}^d(B)\) for some
\(B\in\operatorname{GL}_2(k)\).  If \(B\) has two distinct eigenvalues,
the symmetric-power eigenvalues form a geometric progression.  Cyclicity
of the companion operator forces these eigenvalues to be distinct; hence
\(g_a\) is squarefree and
\(a\in\GRSsurf_{r,k}\cap D(\Disc)\).  If \(B\) has only one eigenvalue, the
scalar case is not cyclic, so \(B\) is nontrivial unipotent up to scalar.
Its symmetric power has characteristic polynomial \((X-\beta)^r\), and
therefore \(a\in\Upower_{r,k}^{\times}\).  The two alternatives are disjoint because
the first is squarefree and the second has a single root of multiplicity
\(r\).
\end{proof}

\begin{lemma}[The non-pure collision boundary is non-MDS]
\label{lem:collision-boundary}
Let \(k\) be any field, let \(r\geq3\), let \(n\geq r+3\), and let
\(\gamma,t\in k^\times\).  Suppose that \(t\) has finite order
\(m\leq r-1\), and put
\[
g(X)=\prod_{i=0}^{r-1}(X-\gamma t^i).
\]
If \(m=1\), then \(t=1\) and \(g=(X-\gamma)^r\).  If \(m\geq2\), the
companion Krylov code \(\mathcal C_g\) is not MDS.
\end{lemma}

\begin{proof}
The assertion for \(m=1\) is immediate, so assume \(m\geq2\).  Write
\[
r=\ell m+s,
\qquad \ell\geq1,
\qquad 0\leq s<m.
\]
Then
\[
g(X)=(X^m-\gamma^m)^\ell
       \prod_{i=0}^{s-1}(X-\gamma t^i).
\]
This factorization is valid in arbitrary characteristic: an element of
exact multiplicative order \(m\) gives all \(m\) distinct roots of
\(X^m-1\).  Any coefficient cancellations in positive characteristic
can only lower the weight estimates below.
If \(s\leq m-2\), its coefficient weight is at most
\[
(\ell+1)(s+1)\leq \ell m+s=r,
\]
because
\(r-(\ell+1)(s+1)=\ell(m-s-1)-1\geq0\).
Thus \(g\) itself is a nonzero multiple of degree less than \(n\) and
weight at most \(r\).

If \(s=m-1\), multiply by the missing factor to obtain
\[
(X-\gamma t^{m-1})g(X)=(X^m-\gamma^m)^{\ell+1}.
\]
This multiple has degree \(r+1<n\) and coefficient weight at most
\(\ell+2\leq r\); the last inequality follows from \(m\geq2\) and
\(\ell\geq1\).  In either case Theorem~\ref{thm:krylov-mds} supplies a
sparse-multiple obstruction to the MDS property.
\end{proof}

\begin{remark}[The collision boundary]
The discriminant exclusion in
Theorem~\ref{thm:geometric-grs-locus} keeps the squarefree semisimple
mechanism separate from the unipotent mechanism.  Although
\(\Upower_{r,k}^{\times}\subseteq\GRSsurf_{r,k}\) through the specialization
\(t=1\), the
pure-power curve belongs to the non-squarefree boundary and must be recorded
separately.  Lemma~\ref{lem:collision-boundary} shows that, in the stable
range, every other collision point with nonzero constant term lies on the
determinantal non-MDS boundary.  We retain the decomposed formula rather
than suppressing this geometric distinction in a compressed set equality.
\end{remark}

\begin{corollary}[The cubic hypersurface and higher codimension]
\label{cor:cubic-surface}
Let \(k\) be an algebraically closed field.  For \(r=3\),
\[
\GRSsurf_{3,k}=V(a_0a_2^3-a_1^3)\subseteq\A_k^3.
\]
For \(r\geq4\),
\[
\operatorname{codim}_{\A_k^r}\GRSsurf_{r,k}=r-2\geq2,
\]
so no single nonzero polynomial defines the complete GRS coefficient
closure set-theoretically.
\end{corollary}

\begin{proof}
Substitution of
\(g(X)=\prod_{i=0}^{2}(X-\gamma t^i)\) gives
\(a_0a_2^3-a_1^3=0\).  To verify irreducibility in every
characteristic, set \(R=k[a_1,a_2]\).  The polynomial
\[
F=a_2^3a_0-a_1^3\in R[a_0]
\]
is primitive because \(\gcd(a_2^3,a_1^3)=1\).  Over
\(\operatorname{Frac}(R)\), it has degree one in \(a_0\), and hence is
irreducible.  Gauss's lemma shows that \(F\) is irreducible in
\(k[a_0,a_1,a_2]\); this argument is unchanged in characteristic three.
Thus \(V(F)\) is an irreducible surface.  The irreducible surface
\(\GRSsurf_{3,k}\) is contained in it, so equality follows.  The assertion
for \(r\geq4\) is the codimension statement of
Theorem~\ref{thm:coefficient-surface}.
\end{proof}

\begin{corollary}[Exact cubic GRS criterion]
\label{cor:exact-cubic-grs}
Let
\[
g(X)=X^3+a_2X^2+a_1X+a_0\in\F_q[X],
\qquad n\geq6.
\]
Then
\[
\mathcal C_g\text{ is MDS and of GRS type}
\quad\Longleftrightarrow\quad
D_{n,3}(g)\ne0
\quad\text{and}\quad
a_0a_2^3-a_1^3=0.
\]
Equivalently, inside the cubic MDS locus,
\[
\mathcal C_g\text{ is of GRS type}
\quad\Longleftrightarrow\quad
a_0a_2^3=a_1^3.
\]
\end{corollary}

\begin{proof}
If \(\mathcal C_g\) is MDS and of GRS type, then
\(D_{n,3}(g)\ne0\) by Theorem~\ref{thm:determinantal-mds}.  After extension
to \(\overline{\F}_q\), Corollary~\ref{cor:absolute-rnc} and
Theorem~\ref{thm:geometric-grs-locus} place the coefficient point of \(g\)
on \(\GRSsurf_{3,\overline{\F}_q}\).  Corollary~\ref{cor:cubic-surface}
therefore gives
\[
a_0a_2^3-a_1^3=0.
\]

Conversely, assume that \(D_{n,3}(g)\ne0\) and
\(a_0a_2^3-a_1^3=0\).  The code is MDS by
Theorem~\ref{thm:determinantal-mds}; hence \(a_0\ne0\) by
Theorem~\ref{thm:krylov-mds}.  Over \(\overline{\F}_q\),
Corollary~\ref{cor:cubic-surface} places the coefficient point on
\(\GRSsurf_{3,\overline{\F}_q}\), while
Theorem~\ref{thm:coefficient-surface}\textup{(5)} supplies
\(\gamma,t\in\overline{\F}_q^\times\) such that
\[
g(X)=\prod_{i=0}^{2}(X-\gamma t^i).
\]
If these roots are distinct, Theorem~\ref{thm:geometric-grs-locus} places
the orbit on an RNC over \(\overline{\F}_q\).  Otherwise,
\(\ord(t)\leq2\).  Lemma~\ref{lem:collision-boundary} excludes
\(\ord(t)=2\), because the code is MDS.  Hence \(t=1\), so \(g\) is a pure
power; Theorem~\ref{thm:geometric-grs-locus} again gives geometric RNC
containment.  Finally, Corollary~\ref{cor:absolute-rnc} descends this
containment and shows that \(\mathcal C_g\) is of GRS type over \(\F_q\).
\end{proof}

\begin{remark}
Corollary~\ref{cor:exact-cubic-grs} gives a single-equation recognition
theorem on the cubic MDS locus.  This is special to \(r=3\): for \(r\geq4\),
the coefficient surface \(\GRSsurf_{r,k}\) has codimension \(r-2\geq2\) in
\(\A_k^r\), and therefore cannot be cut out set-theoretically by a single
nonzero polynomial.
\end{remark}

\subsection{Frobenius descent and exact counting}

The generic-fiber calculation used a transcendental ratio.  Exact counting
over a finite field instead requires a finite-order version with a range
that rules out wraparound coincidences.

\begin{lemma}[Uniqueness up to reversal]
\label{lem:gp-uniqueness}
Let \(d\geq2\), let \(t\) have finite order \(m\geq d+4\), and suppose
\[
\{\gamma t^i:0\leq i\leq d\}
=
\{\delta u^i:0\leq i\leq d\},
\]
where both sets consist of \(d+1\) distinct elements.  Then either
\[
(\delta,u)=(\gamma,t)
\]
or
\[
(\delta,u)=(\gamma t^d,t^{-1}).
\]
\end{lemma}

\begin{proof}
Since \(u\) is the quotient of two elements of the left-hand set, there is
an \(e\in\mathbb Z/m\mathbb Z\) such that \(u=t^e\).  Likewise,
\(\delta=\gamma t^a\) for some \(a\in\mathbb Z/m\mathbb Z\).  Hence
\[
a+e\{0,1,\ldots,d\}=\{0,1,\ldots,d\}
\]
inside \(\mathbb Z/m\mathbb Z\).

Put \(I=\{0,1,\ldots,d\}\), and for \(h\in\mathbb Z/m\mathbb Z\) define
\[
N(h)=\bigl|\{x\in I:x+h\in I\}\bigr|.
\]
The progression \(a,a+e,\ldots,a+de\) has \(d\) consecutive ordered pairs
of difference \(e\), so \(N(e)\geq d\).  For the representative
\(1\leq h\leq m-1\), direct counting gives
\[
N(h)
=
\max\{0,d+1-h\}
+
\max\{0,d+1-(m-h)\}.
\]
Since \(m\geq d+4\), this number equals \(d\) only for \(h=1\) and
\(h=m-1\); for every other \(h\), it is at most \(d-1\).  Therefore
\(e=\pm1\).  If \(e=1\), equality of the two intervals forces \(a=0\).
If \(e=-1\), it forces \(a=d\), proving the assertion.
\end{proof}

\begin{theorem}[Frobenius descent on the coefficient surface]
\label{thm:frobenius-descent}
Let \(r\geq3\), let \(n\geq r+3\), and let
\(g\in\F_q[X]\) be monic, squarefree, and of degree \(r\).  Assume that
\(\mathcal C_g\) is MDS and that the coefficient point of \(g\) belongs to
\[
\GRSsurf_{r,\overline{\F}_q}(\overline{\F}_q).
\]
Then exactly one of the following alternatives occurs.
\begin{enumerate}[label=\textup{(\arabic*)}]
\item There exist \(\gamma,t\in\F_q^\times\) with
      \(\ord(t)\geq n\) such that
      \[
      g(X)=\prod_{i=0}^{r-1}(X-\gamma t^i).
      \]
\item There exist \(\gamma,t\in\F_{q^2}^\times\) with
      \[
      t^{q+1}=1,
      \qquad
      \gamma^q=\gamma t^{r-1},
      \qquad
      \ord(t)\geq n
      \]
      such that
      \[
      g(X)=\prod_{i=0}^{r-1}(X-\gamma t^i).
      \]
\end{enumerate}
In the second case Frobenius reverses the root sequence.  If \(r\) is even,
\(g\) has no \(\F_q\)-root; if \(r\) is odd, it has exactly one
\(\F_q\)-root.
\end{theorem}

\begin{proof}
The MDS condition gives \(a_0\neq0\).  By
Theorem~\ref{thm:coefficient-surface}\textup{(5)}, applied with
\(k=\overline{\F}_q\), there exist
\(\gamma,t\in\overline{\F}_q^\times\) such that
\[
g(X)=\prod_{i=0}^{r-1}(X-\gamma t^i)
\]
over \(\overline{\F}_q\).  If \(t^m=1\) for some \(1\leq m<n\), then
columns whose exponents differ by \(m\) in the root-evaluation
parity-check matrix are proportional, contrary to the MDS hypothesis.
Hence \(\ord(t)\geq n\).  Since
\(n\geq r+3=d+4\), Lemma~\ref{lem:gp-uniqueness} shows that the fiber of
\(\Phi_{r,\overline{\F}_q}\) over \(g\) consists exactly of
\[
(\gamma,t),
\qquad
(\gamma t^{r-1},t^{-1}).
\]

Let \(\sigma(x)=x^q\).  Since \(\Phi_{r,\overline{\F}_q}\) is defined over
\(\F_q\) and \(g\) is \(\F_q\)-rational,
\[
\Phi_{r,\overline{\F}_q}(\gamma^q,t^q)
=\sigma\!\left(\Phi_{r,\overline{\F}_q}(\gamma,t)\right)=g.
\]
Thus Frobenius permutes the two parameter points.  It either fixes
\((\gamma,t)\), in which case
\(\gamma,t\in\F_q^\times\), or exchanges the two, in which case
\[
(\gamma^q,t^q)=(\gamma t^{r-1},t^{-1}).
\]
The exchange equations are equivalent to
\[
t^{q+1}=1,
\qquad
\gamma^q=\gamma t^{r-1},
\]
and also imply \(\gamma,t\in\F_{q^2}^\times\).  Reversal has no fixed
parameter point under the order assumption, so the two alternatives are
exclusive.

In the nonsplit case Frobenius sends the root indexed by \(i\) to the root
indexed by \(r-1-i\).  Since the roots are distinct and
\(\ord(t)\geq n>r-1\), a root is fixed precisely when
\(2i=r-1\).  Such an index exists exactly when \(r\) is odd, and then it
is unique.
\end{proof}

For a positive integer \(M\), let \(C_M\) be a cyclic group of order \(M\)
and put
\[
E_M(n)
=
\#\{t\in C_M:\ord(t)\geq n\}
=
\sum_{\substack{e\mid M\\ e\geq n}}\varphi(e).
\]

Retain the notation \(\mathcal G_{n,r}(q)\) from
Section~\ref{sec:finite-field-classification} for the set of monic
degree-\(r\) polynomials \(g\in\F_q[X]\) for which \(\mathcal C_g\) is MDS
and of GRS type.

\begin{theorem}[Exact number of GRS polynomials]
\label{thm:exact-grs-count}
Let \(r\geq3\), let \(n\geq r+3\), and put
\(p=\operatorname{char}\F_q\).  Then
\[
\boxed{
|\mathcal G_{n,r}(q)|
=
(q-1)\mathbf 1_{\{n\leq p\}}
+
\frac{q-1}{2}
\bigl(E_{q-1}(n)+E_{q+1}(n)\bigr).
}
\]
This is a count of polynomials, not of monomial-equivalence classes of
codes.
\end{theorem}

\begin{proof}
For the split semisimple family, choose
\[
\gamma\in\F_q^\times,
\qquad
t\in\F_q^\times,
\qquad
\ord(t)\geq n.
\]
There are \((q-1)E_{q-1}(n)\) ordered pairs.  Reversal acts by
\[
(\gamma,t)\longmapsto(\gamma t^{r-1},t^{-1})
\]
and preserves the polynomial.  It acts freely because
\(\ord(t)\geq n\geq6\), and
Lemma~\ref{lem:gp-uniqueness} shows that these are the only two parameter
pairs defining the same polynomial.  Hence the split family contributes
\[
\frac{q-1}{2}E_{q-1}(n).
\]

For the nonsplit family, \(t\) ranges over the norm-one cyclic subgroup of
\(\F_{q^2}^\times\), which has order \(q+1\), subject to
\(\ord(t)\geq n\).  For every such \(t\), the equation
\[
\gamma^{q-1}=t^{r-1}
\]
has exactly \(q-1\) solutions in \(\F_{q^2}^\times\): the map
\(x\mapsto x^{q-1}\) surjects onto the norm-one subgroup and has kernel
\(\F_q^\times\).  Reversal is again free and, by
Lemma~\ref{lem:gp-uniqueness}, gives the only duplication.  The nonsplit
family therefore contributes
\[
\frac{q-1}{2}E_{q+1}(n).
\]
The split and nonsplit contributions are disjoint.  Indeed, an overlap of
their root sets would, by Lemma~\ref{lem:gp-uniqueness}, identify the two
ratios up to inversion.  That ratio would then lie both in
\(\F_q^\times\) and in the norm-one subgroup, and so would have order at
most two, contrary to \(\ord(t)\geq n\).

Finally, the pure-power family contains one polynomial
\((X-\beta)^r\) for each \(\beta\in\F_q^\times\).  By
Lemma~\ref{lem:pure-power-mds}, it is MDS precisely when \(n\leq p\), so
this family contributes \((q-1)\mathbf 1_{\{n\leq p\}}\).  Its root
structure is disjoint from both squarefree semisimple families.  Adding the
three contributions proves the formula.
\end{proof}

\subsection{Asymptotic genericity of non-GRS MDS codes}

The coefficient surface has dimension two, whereas the MDS locus is a
nonempty open subset of the full \(r\)-dimensional coefficient space.  The
exact count makes the resulting asymptotic separation quantitative.

\begin{theorem}[Asymptotic genericity of MDS non-GRS codes]
\label{thm:generic-nongrs}
Fix integers \(r\geq3\) and \(n\geq r+3\).  As \(q\to\infty\) through
prime powers,
\[
\#\left\{
g\in\F_q[X]:
\begin{array}{l}
g\text{ is monic of degree }r,\\
\mathcal C_g\text{ is MDS and non-GRS}
\end{array}
\right\}
=q^r+O_{n,r}(q^{r-1}).
\]
Equivalently, the proportion of monic degree-\(r\) polynomials defining MDS
non-GRS codes tends to one.
\end{theorem}

\begin{proof}
Let \(\delta_{n,r}=\deg D_{n,r}\), with \(D_{n,r}\) as in
Theorem~\ref{thm:determinantal-mds}.  By
Proposition~\ref{prop:mds-polynomial-nonzero}, its reduction is nonzero in
every characteristic.  Applying Lemma~\ref{lem:finite-field-zero-bound}
to this reduction gives
\[
\#\{a\in\A^r(\F_q):D_{n,r}(a)=0\}
\leq
\delta_{n,r}q^{r-1}.
\]
By Theorem~\ref{thm:determinantal-mds}, at least
\[
q^r-\delta_{n,r}q^{r-1}
\]
monic degree-\(r\) polynomials define MDS codes.  On the other hand,
Theorem~\ref{thm:exact-grs-count} yields
\[
|\mathcal G_{n,r}(q)|
\leq
(q-1)+\frac{q-1}{2}\bigl((q-1)+(q+1)\bigr)
\leq q^2.
\]
Consequently,
\[
q^r-\delta_{n,r}q^{r-1}-q^2
\leq
\#\{\text{MDS non-GRS members}\}
\leq q^r.
\]
Since \(r\geq3\), the term \(q^2\) is \(O(q^{r-1})\), which proves the
assertion.
\end{proof}
 \section{Principal-ideal realization and selected examples}
\label{sec:principal-ideal-application}

The operator-theoretic construction recovers the original quotient-ring
model without using an ambient modulus in any proof above.

\begin{proposition}[Principal-ideal realization]
\label{prop:principal-ideal-realization}
Let \(g,v\in k[X]\) be monic, with
\(\deg g=r\), \(\deg v=n-r\), and put \(h=gv\).  Under the coefficient
identification \(k[X]_{<n}\simeq k^n\), the principal ideal
\[
 (g)\subseteq k[X]/(h)
\]
is precisely
\[
 \{u(X)g(X):\deg u<n-r\}
 =\mathcal C_g.
\]
Thus the ideal is a quotient-ring realization of the intrinsic companion
Krylov code; the code does not depend on the choice of \(v\).
\end{proposition}

\begin{proof}
Every class in \((g)\) has the form \(gf\bmod gv\).  Replacing \(f\) by
its remainder \(u\) modulo \(v\) gives the unique representative \(gu\)
of degree \(<n\), where \(\deg u<n-r\).  Conversely every such multiple
belongs to the ideal.  Proposition~\ref{prop:companion-realization}
identifies the same coefficient vectors with \(\mathcal C_g\).
\end{proof}

We present the following three concrete examples, corresponding respectively to: a nonsplit GRS orbit, a squarefree MDS non-GRS
orbit invisible to one low-degree obstruction, and a genuine repeated-root
MDS non-GRS orbit. 

\begin{example}[A nonsplit GRS orbit]
\label{ex:nonsplit-grs}
Work over
\[
 \mathbb F_{121}=\mathbb F_{11}(\omega),\qquad \omega^2=2,
\]
and take \(t=8+2\omega\), \(\gamma=7+5\omega\).  Then
\[
 \ord(t)=12,\qquad t^{11}=t^{-1},\qquad
 \gamma^{11}=\gamma t^3.
\]
For \(r=4\) and \(n=7\), the Frobenius pairs in
\(\{\gamma,\gamma t,\gamma t^2,\gamma t^3\}\) give
\[
 g(X)=(X^2+8X+10)(X^2+2X+10)
     =X^4+10X^3+3X^2+X+1.
\]
Both quadratics are irreducible over \(\mathbb F_{11}\).  The product of
all \(4\times4\) minors of \(H_g^{\mathrm{rem}}\) is \(1\), so
\(\mathcal C_g\) is an \([7,3,5]_{11}\) MDS code.  The displayed
Frobenius-reversed progression is family \textup{(N0)} of
Theorem~\ref{thm:arithmetic-classification}; hence the code is of GRS type.
\end{example}

\begin{example}[A squarefree MDS non-GRS orbit with vanishing
low-degree obstruction]
\label{ex:squarefree-nongrs}
Over \(\mathbb F_{11}\), let
\[
 \begin{aligned}
 g(X)
  &=(X^2+X+4)(X^2+3X+3)\\
  &=X^4+4X^3+10X^2+4X+1,
 \qquad n=7.
 \end{aligned}
\]
The two quadratic factors are irreducible, and
\[
 H_g^{\mathrm{rem}}=
 \begin{pmatrix}
 1&0&0&0&10&4&5\\
 0&1&0&0& 7&4&2\\
 0&0&1&0& 1&3&10\\
 0&0&0&1& 7&6&1
 \end{pmatrix}.
\]
The product of its maximal minors is \(5\ne0\), proving that
\(\mathcal C_g\) is MDS.  In family \textup{(N0)}, the two quadratic
Frobenius pairs have the same norm; here their constant terms are \(4\)
and \(3\), so Theorem~\ref{thm:arithmetic-classification} excludes GRS
type.  Nevertheless, for
\[
 \Omega_4(g):=a_0a_2a_3^2-a_1^2a_2
\]
one has \(\Omega_4(g)=0\).  This illustrates geometrically why a single
coefficient equation cannot recognize the codimension-two surface
\(\mathfrak G_4\).
\end{example}

\begin{example}[A genuine repeated-root MDS non-GRS orbit]
\label{ex:repeated-nongrs}
Take
\[
 g(X)=(X-1)^2(X-3)=X^3+6X^2+7X+8
 \quad\text{over }\mathbb F_{11},\qquad n=6.
\]
The Hasse-jet parity-check matrix is
\[
 H_g^{\mathrm{jet}}=
 \begin{pmatrix}
 1&1&1&1&1&1\\
 0&1&2&3&4&5\\
 1&3&9&5&4&1
 \end{pmatrix}.
\]
The product of all \(3\times3\) minors is \(6\ne0\), so the code is
\([6,3,4]_{11}\) MDS.  Its characteristic polynomial is not squarefree
and is not a pure power.  The unipotent alternative in
Theorem~\ref{thm:arithmetic-classification} therefore shows that it is
non-GRS.
\end{example}

The examples emphasize the separation of roles in the present framework:
\(D_{n,r}\) decides the MDS question, while orbit rigidity and the
arithmetic classification decide the GRS question.  The ambient ideal is
only a convenient realization of the resulting coefficient code.
 \section{Concluding remarks and open problems}
\label{sec:conclusion}

The central object of this paper is not the ambient quotient ring but the
cyclic pair \((A,z)\) and its projective Krylov orbit. In the stable range,
the requirement that an MDS orbit segment lie on a rational normal curve
(RNC) is rigid: it forces the acting projectivity to come from the
symmetric-power representation of \(\mathrm{PGL}_2\). For companion
operators, the three projective conjugacy types become the four arithmetic
GRS families over \(\mathbb F_q\).

The coefficient-space picture makes the rigidity visible algebraically.
The semisimple GRS closure is a two-dimensional rational surface,
independently of \(r\), inside the \(r\)-dimensional space of monic
degree-\(r\) polynomials. Its codimension \(r-2\) explains why a single
low-degree obstruction is complete in the cubic case but cannot be complete
in higher degree.  In every characteristic, the normalization of its
nonzero-constant-term part is the affine quotient of the parameter torus by
reversal. Frobenius descent on the generic two-to-one parameterization gives
an exact finite-field count, while the determinantal description of the MDS
open set shows that MDS non-GRS members are asymptotically generic.

The normalization theorem shifts the remaining geometric questions from
the existence of a torus quotient to its local and embedded geometry.  A
first problem is to determine the conductor and the non-isomorphism locus of
the normalization morphism
\[
\mathbb G_{m,k}^{2}/\!/\langle\iota\rangle
\longrightarrow
\mathfrak G_{r,k}\cap D(a_0)
\]
and to relate them to the collision strata and the determinantal non-MDS
boundary.  In characteristic two, the quotient morphism is finite
\(\acute{e}\)tale off the fixed curve \(t=1\), while the fixed scheme is
nonreduced when \(r\) is odd; determining the different and the resulting
local structure along its image remains open.  Further problems are to
determine the defining ideal of \(\mathfrak G_4\), and more generally the
degrees and singular loci of \(\mathfrak G_r\).  On the coding side, one can
ask for monomial-equivalence classes rather than generator polynomials.
 
\appendix
\section{Boundary dimensions}
\label{app:boundary-cases}

The rigidity theorem concerns \(r\ge3\) and \(n\ge r+3\).  For
completeness, this appendix records the complementary dimensions.  Write
\(\mathcal C_g=\mathcal K_n(T_g,\bar1)\).
For \(r=1\), where the RNC formulation degenerates, ``of GRS type'' is
understood in the conventional (possibly extended) evaluation-code sense.

\begin{theorem}[Complementary parameter ranges]
\label{thm:boundary-ranges}
Let \(1\le r<n\), and let
\[
 g(X)=X^r+a_{r-1}X^{r-1}+\cdots+a_1X+a_0\in\mathbb F_q[X]
\]
be monic.
\begin{enumerate}[label=\textup{(\roman*)}]
\item If \(r=1\), write \(g(X)=X-\beta\).  Then \(\mathcal C_g\) is
      MDS exactly when \(\beta\ne0\).  In that case it is of GRS type
      exactly when \(n\le q+1\).
\item Suppose \(r=2\), and put \(p=\operatorname{char}\mathbb F_q\).
      If \(g\) has distinct geometric roots \(\alpha,\beta\), then
      \(\mathcal C_g\) is MDS exactly when
      \[
      \alpha\beta\ne0,\qquad \ord(\beta/\alpha)\ge n.
      \]
      If \(g=(X-\beta)^2\), then it is MDS exactly when
      \(\beta\ne0\) and \(n\le p\).  Every MDS code in either case is of
      GRS type.
\item If \(r\ge3\) and \(n=r+1\), then \(\mathcal C_g\) is MDS exactly
      when \(a_0,\ldots,a_{r-1}\) are all nonzero.  It is then of GRS type
      exactly when \(n\le q+1\).
\item If \(r\ge3\) and \(n=r+2\), put \(a_r=1\).  Then \(\mathcal C_g\)
      is MDS exactly when all \(a_0,\ldots,a_r\) are nonzero and
      \[
      \frac{a_1}{a_0},\frac{a_2}{a_1},\ldots,
      \frac{a_r}{a_{r-1}}
      \]
      are pairwise distinct.  Every such MDS code is of GRS type.
\end{enumerate}
\end{theorem}

\begin{proof}
For \(r=1\), reduction modulo \(X-\beta\) gives the parity-check row
\[
 (1,\beta,\beta^2,\ldots,\beta^{n-1}).
\]
It has full support exactly when \(\beta\ne0\).  A full-support
one-dimensional code is an extended \(\operatorname{GRS}_1\) code precisely
when \(n\) distinct points can be chosen on
\(\mathbb P^1(\mathbb F_q)\), that is, when \(n\le q+1\); duality gives the
claim for \(\mathcal C_g\).

For \(r=2\) with distinct roots, a parity-check matrix over the splitting
field has columns \((\alpha^i,\beta^i)^{\mathsf T}\).  The determinant in
columns \(i<j\) is
\[
 \alpha^i\beta^i(\beta^{j-i}-\alpha^{j-i}),
\]
and all such determinants are nonzero precisely under the stated order
condition.  For \(g=(X-\beta)^2\), Hasse jets, followed by nonzero row and
column scalings, give
\[
 \begin{pmatrix}
 1&1&\cdots&1\\
 0&1&\cdots&n-1
 \end{pmatrix}.
\]
Its columns are pairwise independent exactly when \(\beta\ne0\) and the
images of \(0,\ldots,n-1\) in the prime field are distinct, equivalently
\(n\le p\).  Every two-dimensional MDS parity-check configuration is a
set of distinct points of \(\mathbb P^1(\mathbb F_q)\), hence is of GRS
type.

If \(n=r+1\), the code is generated by
\[
 (a_0,a_1,\ldots,a_{r-1},1),
\]
which is MDS exactly when it has full support.  The first case and duality
give the GRS statement.  If \(n=r+2\), a generator matrix is
\[
 \begin{pmatrix}
 a_0&a_1&\cdots&a_r&0\\
 0&a_0&\cdots&a_{r-1}&a_r
 \end{pmatrix}.
\]
It is MDS precisely when every column is nonzero and no two columns are
proportional.  The first requirement says \(a_i\ne0\) for every \(i\);
the affine slopes of the middle columns are the displayed ratios, giving
the second requirement.  These columns form a subset of
\(\mathbb P^1(\mathbb F_q)\), so the MDS code is of GRS type.
\end{proof}
 \section{The stabilizer of a rational normal curve}
\label{app:stabilizer}

Let \(k\) be a field, let \(d\ge2\), and write
\[
 \boldsymbol\nu_d(s,t)
  =(s^d,s^{d-1}t,\ldots,st^{d-1},t^d)^{\mathsf T},
 \qquad
 \Gamma_0=\nu_d(\PP^1_k)\subseteq\PP^d_k.
\]
Our convention for \(\Sym^d(B)\) is determined by
\[
 \boldsymbol\nu_d\!\left(B(s,t)^{\mathsf T}\right)
   =\Sym^d(B)\boldsymbol\nu_d(s,t).
\]
Every stabilizer below is the stabilizer of the geometric curve, not
merely of its finite set of \(k\)-rational points.

\begin{proposition}[The projective stabilizer of an RNC]
\label{prop:rnc-stabilizer}
The symmetric-power representation induces an injective homomorphism
\[
 \overline{\Sym}^d:\operatorname{PGL}_2(k)
   \longrightarrow\operatorname{PGL}_{d+1}(k),
\]
and
\[
 \operatorname{Stab}_{\operatorname{PGL}_{d+1}(k)}(\Gamma_0)
   =\overline{\Sym}^d\bigl(\operatorname{PGL}_2(k)\bigr).
\]
The same statement holds after every extension of the ground field.
\end{proposition}

\begin{proof}
Equivariance of \(\nu_d\) gives the inclusion from right to left.
Scalar matrices in \(\operatorname{GL}_2(k)\) act through scalar
matrices on \(\Sym^d(k^2)\), so the action descends to
\(\operatorname{PGL}_2(k)\).

Conversely, let
\[
 [M]\in\operatorname{Stab}_{\operatorname{PGL}_{d+1}(k)}(\Gamma_0).
\]
Since \(\nu_d:\PP^1_k\to\Gamma_0\) is a \(k\)-isomorphism, restriction
of \([M]\) induces
\[
 \beta=\nu_d^{-1}\circ[M]\circ\nu_d
     \in\operatorname{Aut}_k(\PP^1)
     =\operatorname{PGL}_2(k).
\]
Choose \(B\in\operatorname{GL}_2(k)\) representing \(\beta\).  Then
\([M]\) and \([\Sym^d(B)]\) induce the same automorphism of
\(\Gamma_0\), so
\[
 [N]=[M][\Sym^d(B)]^{-1}
\]
fixes every geometric point of \(\Gamma_0\).

It remains to see that \([N]\) is the identity without imposing a
lower bound on \(\#k\).  Choose \(d+2\) pairwise distinct elements
\[
 u_0,\ldots,u_{d+1}\in\overline{k}
\]
and put \(P_i=\nu_d([1:u_i])\).  Any \(d+1\) of these points are
linearly independent, because the corresponding coordinate
determinant is the nonzero Vandermonde product
\[
 \prod_{a<b}(u_b-u_a).
\]
Thus \(P_0,\ldots,P_{d+1}\) form a projective frame in
\(\PP^d_{\overline{k}}\).  A projectivity fixing a projective frame
pointwise is the identity: after the first \(d+1\) frame points are
used as coordinate points, a representative matrix is diagonal, and
the last frame point, whose coordinates are all nonzero, forces all
diagonal entries to agree.  Hence \([N]=1\), and
\([M]=[\Sym^d(B)]\).

This argument also proves injectivity, since an element in the kernel
acts trivially on \(\nu_d(\PP^1)\). The frame is chosen over
\(\overline{k}\), so no lower bound on \(\#k\) is required.
\end{proof}

\begin{lemma}[Spectrum of a symmetric-power projectivity]
\label{lem:symmetric-power-spectrum}
Let \(k=\F_q\), let \(B\in\operatorname{GL}_2(\F_q)\), let
\(c\in\F_q^\times\), and put
\[
 A=c\,\Sym^d(B).
\]
If \(\lambda,\mu\in\F_{q^2}^\times\) are the eigenvalues of \(B\),
counted with algebraic multiplicity, then the eigenvalues of \(A\) are
\[
 c\lambda^d,\ c\lambda^{d-1}\mu,\ \ldots,\
 c\lambda\mu^{d-1},\ c\mu^d.
\]
In particular, all eigenvalues of \(A\) lie in \(\F_{q^2}\), and the
three nontrivial projective conjugacy types have the following
behavior.
\begin{enumerate}[label=\textup{(\roman*)}]
\item In the split semisimple case,
      \(\lambda,\mu\in\F_q^\times\) are distinct and \(A\) is
      diagonalizable over \(\F_q\).
\item In the nonsplit semisimple case,
      \(\mu=\lambda^q\) with
      \(\lambda\in\F_{q^2}\setminus\F_q\), and \(A\) is
      diagonalizable over \(\F_{q^2}\).  No diagonalization over
      \(\F_q\) is asserted.
\item In the unipotent case, \(B\) has a single eigenvalue
      \(\lambda\in\F_q^\times\), and every eigenvalue of \(A\) equals
      \(c\lambda^d\).  No semisimplicity assertion is made.
\end{enumerate}
Consequently, if \(A\) has at least two distinct eigenvalues, then it
is diagonalizable over \(\F_{q^2}\), and its minimal polynomial over
\(\overline{\F}_q\) is squarefree.
\end{lemma}

\begin{proof}
The characteristic polynomial of \(B\) splits over \(\F_{q^2}\).
After a change of basis over that field, \(B\) is upper triangular:
\[
 B\sim
 \begin{pmatrix}
  \lambda&\eta\\
  0&\mu
 \end{pmatrix}.
\]
On the ordered monomial basis
\[
 s^d,s^{d-1}t,\ldots,st^{d-1},t^d,
\]
its \(d\)-th symmetric power is triangular with diagonal entries
\[
 \lambda^d,\lambda^{d-1}\mu,\ldots,
 \lambda\mu^{d-1},\mu^d.
\]
Multiplication by \(c\) gives the displayed spectrum.

If \(\lambda\ne\mu\), then \(B\) is diagonalizable over its splitting
field, and applying \(\Sym^d\) to an eigenbasis diagonalizes \(A\)
there as well.  The splitting field is \(\F_q\) in the split case and
is \(\F_{q^2}\) in the nonsplit case.  If
\(\lambda=\mu\), every displayed eigenvalue is \(c\lambda^d\); this
includes the nontrivial unipotent case (and the projective identity).
Finally, two distinct eigenvalues of \(A\) force
\(\lambda\ne\mu\), so the preceding diagonalization applies.  The
argument uses only triangularization and therefore is valid in every
characteristic, including when \(\operatorname{char}\F_q\le d\).
\end{proof}

\section*{Funding}
This work was supported by the National Natural Science Foundation of China
[grant numbers 12571003 and 12501006] and the Basic and Applied Basic
Research Foundation of Guangdong Province [grant number 2024A1515010589].

\bibliographystyle{amsplain}

\begin{thebibliography}{10}

\bibitem{AbdukhalikovDingVerma2026}
Kanat Abdukhalikov, Cunsheng Ding, and Gyanendra~K. Verma, \emph{Some
  constructions of non-generalized {Reed--Solomon} {MDS} codes}, Discrete
  Mathematics \textbf{349} (2026), no.~10, 115202,
  \url{https://doi.org/10.1016/j.disc.2026.115202}.

\bibitem{Ball-Grassl-Rotteler}
Simeon Ball, \emph{{Grassl--R{\"o}tteler} cyclic and consta-cyclic {MDS} codes
  are generalised {Reed--Solomon} codes}, Designs, Codes and Cryptography
  \textbf{91} (2023), no.~5, 1685--1694,
  \url{https://doi.org/10.1007/s10623-022-01174-5}.

\bibitem{Ball-Lavrauw}
Simeon Ball and Michel Lavrauw, \emph{Arcs in finite projective spaces}, EMS
  Surveys in Mathematical Sciences \textbf{6} (2019), no.~1--2, 133--172,
  \url{https://doi.org/10.4171/EMSS/33}.

\bibitem{Beelen-Glynn-Hoholdt-Kaipa}
Peter Beelen, David~G. Glynn, Tom H{\o}holdt, and Krishna~V. Kaipa,
  \emph{Counting generalized {Reed--Solomon} codes}, Advances in Mathematics of
  Communications \textbf{11} (2017), no.~4, 777--790,
  \url{https://doi.org/10.3934/amc.2017057}.

\bibitem{Beelen-Puchinger-Rosenkilde}
Peter Beelen, Sven Puchinger, and Johan Rosenkilde, \emph{Twisted
  {Reed--Solomon} codes}, IEEE Transactions on Information Theory \textbf{68}
  (2022), no.~5, 3047--3061, \url{https://doi.org/10.1109/TIT.2022.3146254}.

\bibitem{CaminataEtAl2018}
Alessio Caminata, Noah Giansiracusa, Han-Bom Moon, and Luca Schaffler,
  \emph{Equations for point configurations to lie on a rational normal curve},
  Advances in Mathematics \textbf{340} (2018), 653--683,
  \url{https://doi.org/10.1016/j.aim.2018.10.013}.

\bibitem{ChenHuangWu2025}
Bocong Chen, Jing Huang, and Hao Wu, \emph{Regular cyclic $(q+1)$-arcs in
  $\mathrm{PG}(3,2^m)$: spectral rigidity, descent, and an {MDS} criterion},
  arXiv preprint arXiv:2512.19371, December 2025, Version 1, 22 December 2025;
  \url{https://arxiv.org/abs/2512.19371}.

\bibitem{Flowe-Harris}
Randolph~P. Flowe and Gary~A. Harris, \emph{A note on generalized {Vandermonde}
  determinants}, SIAM Journal on Matrix Analysis and Applications \textbf{14}
  (1993), no.~4, 1146--1151, \url{https://doi.org/10.1137/0614079}.

\bibitem{Harris1992}
Joe Harris, \emph{Algebraic geometry: A first course}, Graduate Texts in
  Mathematics, vol. 133, Springer-Verlag, New York, 1992,
  \url{https://doi.org/10.1007/978-1-4757-2189-8}.

\bibitem{Hirschfeld-Thas}
J.~W.~P. Hirschfeld and J.~A. Thas, \emph{General {Galois} geometries}, second
  ed., Springer Monographs in Mathematics, Springer-Verlag, London, 2016,
  \url{https://doi.org/10.1007/978-1-4471-6790-7}.

\bibitem{Jin-Ma-Xing-Zhou}
Lingfei Jin, Liming Ma, Chaoping Xing, and Haiyan Zhou, \emph{New families of
  non-{Reed--Solomon} {MDS} codes}, IEEE Transactions on Information Theory
  \textbf{72} (2026), no.~2, 985--993,
  \url{https://doi.org/10.1109/TIT.2025.3647222}.

\bibitem{Joyner-Ksir-Traves}
David Joyner, Amy Ksir, and Will Traves, \emph{Automorphism groups of
  generalized {Reed--Solomon} codes}, Advances in Coding Theory and
  Cryptography (Tanush Shaska, W.~Cary Huffman, David Joyner, and Vladimir
  Ustimenko, eds.), Series on Coding Theory and Cryptology, vol.~3, World
  Scientific, Hackensack, NJ, 2007,
  \url{https://doi.org/10.1142/9789812772022_0008}, pp.~114--125.

\bibitem{Li-Lin}
Chi-Kwong Li and Jephian C.-H. Lin, \emph{Confluent {Vandermonde} matrix and
  related topics}, Elemente der Mathematik (2026), Published online first,
  Published 18 February 2026; \url{https://doi.org/10.4171/EM/570}.

\bibitem{Liu-Liu-constacyclic}
Hongwei Liu and Shengwei Liu, \emph{A class of constacyclic codes are
  generalized {Reed--Solomon} codes}, Designs, Codes and Cryptography
  \textbf{91} (2023), no.~12, 4143--4151,
  \url{https://doi.org/10.1007/s10623-023-01294-6}.

\bibitem{LiuLiuOggier2026}
Shengwei Liu, Hongwei Liu, and Fr{\'e}d{\'e}rique~E. Oggier,
  \emph{Constructions of non-generalized {Reed--Solomon} {MDS} codes}, IEEE
  Transactions on Information Theory \textbf{72} (2026), no.~4, 2112--2122,
  \url{https://doi.org/10.1109/TIT.2026.3664947}.

\bibitem{MacWilliams-Sloane}
F.~Jessie MacWilliams and N.~J.~A. Sloane, \emph{The theory of error-correcting
  codes}, North-Holland Mathematical Library, vol.~16, North-Holland Publishing
  Company, Amsterdam--New York--Oxford, 1977, Parts I--II; ISBN
  978-0-444-85009-6 and 978-0-444-85010-2.

\bibitem{Maruta1997}
Tatsuya Maruta, \emph{Cyclic arcs and pseudo-cyclic {MDS} codes}, Discrete
  Mathematics \textbf{174} (1997), no.~1--3, 199--205,
  \url{https://doi.org/10.1016/S0012-365X(96)00334-2}.

\bibitem{Reed-Solomon}
Irving~S. Reed and Gustave Solomon, \emph{Polynomial codes over certain finite
  fields}, Journal of the Society for Industrial and Applied Mathematics
  \textbf{8} (1960), no.~2, 300--304, \url{https://doi.org/10.1137/0108018}.

\bibitem{Shi-Li-Sepasdar-Sole}
Minjia Shi, Xiaoxiao Li, Zahra Sepasdar, and Patrick Sol{\'e}, \emph{Polycyclic
  codes as invariant subspaces}, Finite Fields and Their Applications
  \textbf{68} (2020), 101760, \url{https://doi.org/10.1016/j.ffa.2020.101760}.

\bibitem{Shokrollahi2000}
M.~Amin Shokrollahi, \emph{On cyclic {MDS}-codes}, Coding Theory and
  Cryptography: From Enigma and Geheimschreiber to Quantum Theory (Berlin,
  Heidelberg) (David Joyner, ed.), Springer-Verlag, 2000,
  \url{https://doi.org/10.1007/978-3-642-59663-6_11}, pp.~202--213.

\bibitem{StacksProject}
{The Stacks Project Authors}, \emph{The {Stacks Project}},
  \url{https://stacks.math.columbia.edu}, 2026, Tags \texttt{0BRI} and
  \texttt{07S7}; accessed 14 July 2026.

\bibitem{WangLiuLuo2026}
Guodong Wang, Hongwei Liu, and Jinquan Luo, \emph{New constructions of
  non-{GRS} {MDS} codes, recovery and determination algorithms for {GRS}
  codes}, IEEE Transactions on Information Theory \textbf{72} (2026), no.~6,
  3848--3861, \url{https://doi.org/10.1109/TIT.2026.3683741}.

\end{thebibliography}
\providecommand{\bysame}{\leavevmode\hbox to3em{\hrulefill}\thinspace}
\providecommand{\MR}{\relax\ifhmode\unskip\space\fi MR }
\providecommand{\MRhref}[2]{%
  \href{http://www.ams.org/mathscinet-getitem?mr=#1}{#2}
}
\providecommand{\href}[2]{#2}

\end{document}